\RequirePackage{snapshot}

\newif\ifsubsections

\subsectionstrue 

\documentclass[]{pcmi}

\usepackage[sc]{mathpazo}          
\usepackage{eulervm}               
\usepackage[scaled=0.86]{berasans} 
\usepackage[scaled=1]{inconsolata} 
\usepackage[T1]{fontenc}

\usepackage[%
	protrusion=true,
	expansion=false,
	auto=false
	]{microtype}

\usepackage{xcolor}
\usepackage{graphicx}
\graphicspath{{./figures/}}

\ifdraft
	\definecolor{linkred}{rgb}{0.7,0.2,0.2}
	\definecolor{linkblue}{rgb}{0,0.2,0.6}
\else
	\definecolor{linkred}{rgb}{0.0,0.0,0.0}
	\definecolor{linkblue}{rgb}{0,0.0,0.0}
\fi

\usepackage{stmaryrd,mathrsfs,bm,amsthm,mathtools,yfonts,amssymb,color}

\PassOptionsToPackage{hyphens}{url} 
\usepackage[
    setpagesize=false,
    pagebackref,
	pdfpagelabels=false,
    pdfstartview={FitH 1000},
    bookmarksnumbered=false,
    linktoc=all,
    colorlinks=true,
    anchorcolor=black,
    menucolor=black,
    runcolor=black,
    filecolor=black,
    linkcolor=linkblue,
	citecolor=linkblue,
	urlcolor=linkred,
]{hyperref}
\usepackage[backrefs,msc-links,nobysame]{amsrefs}

\customizeamsrefs

\theoremstyle{plain}
\newtheorem{mythm}[equation]{Theorem}
\newtheorem{myprop}[equation]{Proposition}
\newtheorem{mylem}[equation]{Lemma}
\newtheorem{mycor}[equation]{Corollary}

\theoremstyle{definition}
\newtheorem{mydef}[equation]{Definition}

\newcommand\gr{\mathrm{gr}\, }
\newcommand\Lip{\mathrm{Lip}\,}
\newcommand\Vol{\mathrm{Vol}}
\newcommand\bB{\mathbf{B}}
\newcommand\bE{\mathbf{E}}
\newcommand\dist{\mathrm{dist}\,}
\newcommand\bp{{\mathbf{p}}}
\newcommand\bC{{\mathbf{C}}}

\begin{document}

\title{Almgren's center manifold in a simple setting }

\author{Camillo De Lellis}
\address{Institute for Advanced Study, 1 Einstein Drive, Princeton, NJ 08540, USA}
\address{Institut f\"ur Mathematik, Universit\"at Z\"urich, CH-8057 Z\"urich, Switzerland}
\email{camillo.delellis@math.uzh.ch}


\begin{abstract}
We aim at explaining the most basic ideas underlying two fundamental results in the regularity theory of area minimizing oriented surfaces: De Giorgi's celebrated $\varepsilon$-regularity theorem and Almgren's center manifold. Both theorems will be proved in a very simplified situation, which however allows to illustrate some of the most important PDE estimates. 
\end{abstract}

\maketitle
\thispagestyle{empty}

\section{Introduction}


In these lecture notes I will try to give the core ideas of two fundamental regularity results in geometric measure theory. The
subject is rather technical and complicated and it would require at least one monographic semester course of prerequisites before one could even start with the statements. Nonetheless the core of the arguments have a simple analytic (in modern terms ``PDE'') nature. These notes are an attempt of conveying them without requiring any knowledge of geometric measure theory. 

\subsection{Area minimizing graphs} 
Throughout these notes we will thus fix our attention on graphs of Lipschitz maps $u : \Omega \subset \mathbb R^m \to \mathbb R^n$ 
with Lipschitz constant $\Lip (u)$. In particular given any Borel set $F\subset \Omega$ we will denote by $\gr (u, F)$ the set
\[
\gr (u, F) := \{(x,y)\in F\times \mathbb R^n : y = u(x)\}\, .
\]
We will sometimes omit $F$ if it does not play an important role in our discussion.

We will often use the simple observation that, if we rotate our system of coordinates by a small angle $\theta$, $\gr (u, \Omega)$ is still
the graph of a Lipschitz function over some domain $\Omega'$ in the new coordinates. More precisely we have the following simple
lemma.

\begin{mylem}\label{l:first_rotation}
There are (dimensional) constants $c_0>0$ and $C>0$ with the following property. Assume $\Omega$ and $u$ are as above
with $\Lip (u) \leq 2$
and let $A\in SO (m+n)$ with $|A - {\rm Id}| \leq c_0$\footnote{From now on $|B|$ will denote the Hilbert-Schmidt norm of the matrix $B$.} (where ${\rm Id}$ is the identity map). If we define the coordinate transformation $(x',y') = A (x,y)$, 
then there is $\Omega'\subset \mathbb R^m$ and a Lipschitz $u': \Omega'\to \mathbb R^n$ such that
\[
\{(x',y'): A^{-1} (x',y')\in \gr (u, \Omega)\} = \gr (u', \Omega')
\] 
and $\Lip (u') \leq \Lip (u) + C |A- {\rm Id}|$. 
\end{mylem}

As it is customary, for $F$ Borel, we will let $\Vol^m (\gr (u, F))$ be the $m$-dimensional Hausdorff measure of $\gr (u,F)$, for which the area formula gives the
following identity
\begin{equation}\label{e:area_formula}
\Vol^m (\gr (u, F)) = \int_F \Big(1+|Du|^2 + \sum_{|\alpha|\geq 2} (\det M_{\alpha\beta} (Du))^2\Big)^{\frac{1}{2}}\, ,
\end{equation}
Here we use the notation
$M_{\alpha\beta} (D h)$ for the $k\times k$ minor of $Dh$ corresponding to the choice of the $\alpha_1, \ldots, \alpha_k$ lines and
$\beta_1, \ldots, \beta_k$ rows and we set $|\alpha|:=k$. As it is obvious from the invariance of the Hausdorff
measure under rotations, $\Vol^m (\gr (u, \Omega)) = \Vol^m (\gr (u', \Omega'))$ when $\Omega, \Omega'$, $u$ and $u'$ are as in Lemma \ref{l:first_rotation} and $\Omega$ is Borel (it is elementary to see that then $\Omega'$ is Borel as well). 

In the rest of the paper we will investigate maps $u$ whose graphs are area minimizing in the following sense:

\begin{mydef}\label{d:AM}
Let $u$ and $c_0$ be as in Lemma \ref{l:first_rotation}. We say that $\gr (u, \Omega)$ is area minimizing when the following holds for every $A\in SO (m+n)$ with $|A - {\rm Id}|\leq c_0$:
\begin{itemize} 
\item[(AM)] If $u'$ is the map of Lemma \ref{l:first_rotation} and $v: \Omega' \to \mathbb R^m$ any other Lipschitz map with $\{v\neq u'\} \subset \subset \Omega'$, then
\[
\Vol^m (\gr (v, \Omega')) \geq \Vol^m (\gr (u', \Omega')) = \Vol^m (\gr (u, \Omega))\, .
\]
\end{itemize}
\end{mydef}

\subsection{De Giorgi's $\varepsilon$-regularity theorem.}

It is well known that area minimizing graphs are in fact real analytic if the Lipschitz constant is sufficiently small \footnote{While the restriction on the Lipschitz constant is unnecessary when $n=1$ (where one can use, for instance, the celebrated De Giorgi-Nash theorem), for $n>1$ the situation is much more complicated and it is for instance possible to construct Lipschitz minimal graphs which are not $C^1$, cf. \cite{Lawson-Osserman}. The latter examples are, however, {\em critical points} and not absolute minimizers.}. A ``classical'' path to the statement above is to prove first that $u$ is $C^{1,\alpha}$ and then use Schauder estimates for the Euler-Lagrange equation satisfied by $u$ (which is in fact an elliptic system of partial differential equations) to 
show that $u$ has higher regularity. 

The first step is a corollary of a celebrated theorem by De Giorgi. An appropriately general framework for its statement would be that of area minimizing integer rectifiable currents, which however would require the introduction of a lot of terminology and technical tools from geometric measure theory. The first goal of these notes is thus to illustrare De Giorgi's key idea in the simplified setting of graphs.

\begin{mythm}[De Giorgi]\label{t:de_giorgi}
For every $0<\alpha <1$ there are geometric constants $\varepsilon_0, C>0$ depending only on $\alpha$, $m$ and $n$ with the following property.
Let $\Omega = B_1\subset \mathbb R^m$ and $u: B_1 \to \mathbb R^n$ be a Lipschitz map with $\Lip (u) \leq 1$ whose graph is area minimizing.
Assume
\begin{equation}\label{e:excess}
E := \Vol^m (\gr (u, B_1)) - \omega_m < \varepsilon_0\, ,
\end{equation}
where $\omega_m$ denotes the $m$-dimensional volume of the unit disk\footnote{In these notes we will use the term {\em disk} for 
$B_r (x) := \{y\in \mathbb R^m : |y-x|<r\}$  and the term {\em ball} for $\bB_r (p) := \{q\in \mathbb R^{m+n} : |q-p|<r\}$.} $B_1 = B_1 (0) \subset \mathbb R^m$.
Then $u\in C^{1,\alpha} (B_{1/2})$ and in fact $\|Du\|_{C^{0,\alpha} (B_{1/2})} \leq C E^{\frac{1}{2}}$.
\end{mythm}
Observe that obviously the quantity $E$ is nonnegative and that it equals $0$ if and only if the function $u$ is constant: $E$ measures thus how close is the surface $\gr (u, B_1)$ to be an horizontal disk $B_1 (0) \times \{y\}$.

\subsection{A formula for the excess.}

De Giorgi proved his theorem in \cite{De_Giorgi} in codimension 1 (namely $n=1$) in the framework of reduced boundaries of sets of finite perimeter (which is equivalent to the setting of codimension 1 integral currents, see \cite{Federer} or \cite{Simon}). The statement was then generalized to higher codimension (and to minimizers of a general elliptic integrand) in the framework of integral currents by Almgren in \cite{Almgren68}. In such generality De Giorgi's theorem says that if at a certain scale the mass of an area minimizing current is not much larger than that of a disk of the same diameter, then the current is in fact a $C^{1,\alpha}$ graph (at a slightly smaller scale). 
The interested reader can consult the survey article \cite{DL_CDM} for a quick and not (too) technical intoduction to the topic.

Before proceeding further we want to highlight an important computation which shows how $E$ in \eqref{e:excess} is essentially an $L^2$ measure of the flatness of $\gr (u, B_1)$. More precisely consider the standard basis $e_1, \ldots, e_m, e_{m+1}, \ldots , e_{m+n}$ and let 
\[
\vec{\pi}_0 := e_1 \wedge \ldots \wedge e_m
\] 
be the standard unit $m$-vector orienting $\pi_0 = \mathbb R^m \times \{0\}$. If $x$ is a point of differentiability of $u$ and
$T_p \gr (u)$ is the tangent space to $\gr (u)$ at $p = (x, u(x))$, it is then possible to give a standard orientation to it using the $m$-vector
\[
\vec{T}_p \gr (u) := \frac{(e_1 + du|_p (e_1)) \wedge \ldots \wedge (e_m + du|_p (e_m))}{|(e_1 + du|_p (e_1)) \wedge \ldots \wedge (e_m + du|_p (e_m))|}\, .
\]
In the formula above $du|_p (e_j)$ denotes the following vector of $\{0\}\times \mathbb R^n \subset \mathbb R^{m+n}$:
\[
du|_p (e_j) = \sum_{k=1}^n \frac{\partial u_k}{\partial x_j} (x)\, e_{m+k}\, . 
\]
Moreover, we endow, as customarily, the space of $m$-vectors with a standard euclidean scalar product, which on simple $m$-vectors reads as
\[
\langle v_1\wedge\ldots \wedge v_m , w_1\wedge \ldots \wedge w_m\rangle = \det (\langle v_j, w_k\rangle)\, .
\]
In particular $|\vec{v}| = \sqrt{\langle \vec{v}, \vec{v}\rangle}$ is the induced euclidean norm.

Elementary computations give then the following identity, which we leave as an exercise (in the rest of the notes we use the notation $|\Omega|$ for the Lebesgue $m$-dimensional measure of $\Omega \subset \mathbb R^m$ and we denote by $\Vol^m$ the Hausdorff $m$-dimensional measure on $\mathbb R^{m+n}$).

\begin{myprop}\label{p:cil-sfer}
Let $\Omega\subset \mathbb R^m$ be Borel and $u: \Omega \to \mathbb R^n$ be a Lipschitz map. Then
\begin{equation}\label{e:formula_excess}
\Vol^m (u, \Omega) - |\Omega| = \frac{1}{2} \int_{\gr (u, \Omega)} |\vec{T}_p \gr (u) - \vec{\pi}_0|^2\, d\Vol^m (p)\, . 
\end{equation}
\end{myprop}

For convenience we stop our discussion to introduce a quantity which will play a fundamental role in the rest of our investigations.

\begin{mydef} Let $u: B_r (x) \to \mathbb R^n$ be a Lipschitz map and $\vec{\pi}$ a unit $m$-vector orienting the plane $\pi$. The {\em cylindrical excess} of $\gr (u)$ in 
$\bC_r (x) := B_r (x)\times \mathbb R^n$ with respect to the (oriented) plane $\pi$ is then given by
\begin{equation}\label{e:cylindrical}
\bE (\gr (u), \bC_r (x), \vec\pi) := \frac{1}{2} \int_{\gr (u, B_r (x))} |\vec{T}_p \gr (u) - \vec{\pi}|^2\, d\Vol^m (p)\, .
\end{equation}
If $\vec{\pi} = \vec{\pi}_0$ is the unit $m$-vector which gives to $\mathbb R^m\times \{0\}$ the standard orientation, we then write
$\bE (\gr (u), \bC_r (x))$ in place of $\bE (\gr (u), \bC_r (x), \vec\pi_0)$. 
\end{mydef}

In what follows, if $p= (x,y)\in \mathbb R^m\times \mathbb R^n$ we will often, by a slight abuse of notation, write $\bC_r (p, \vec\pi)$ (resp. $\bC_r (p)$) in place of $\bC_r (x, \vec\pi)$ (resp. $\bC_r (x)$).

\subsection{Codimension 1 and higher codimension.}\label{ss:branching} Observe that, for a general oriented surface $\Sigma$ with no boundary in the cylinder $\Omega \times \mathbb R^n$, the right hand side of \eqref{e:formula_excess} can be small even if the left hand side is fairly large. Indeed, if $\Sigma$ consists of $N$ parallel horizontal planes with the same orientation of $\pi_0$, the right hand side of \eqref{e:formula_excess} is zero, whereas the left hand side is $(N-1) |\Omega|$. The example is, moreover, area minimizing. 

In codimension $1$ De Giorgi's regularity theorem can be considerably strengthened in the following sense. Consider an oriented surface $\Sigma$ with no boundary in $\bC_1 (0) \subset \mathbb R^m \times \mathbb R$, with volume bounded by some constant $M$ and which locally minimizes the area. If 
\begin{equation}\label{e:excess10}
\frac{1}{2} \int_{\Sigma} |\vec{T}_p \Sigma - \vec{\pi}_0|^2 =: \mathbf{E}
\end{equation}
is sufficiently small (depending only upon $n$, $m$ and $M$), then $\Sigma \cap \bC_{1/2} (0)$ consists of finitely many disjoint $C^{1,\alpha}$ graphs over $B_{1/2}$ (and obviously the number $N$ of such graphs can be bounded by $\omega_m^{-1} 2^m M$). Again, the above fact can be conveniently stated and proved in the framework of integral currents. 

In higher codimension the latter version of De Giorgi's $\varepsilon$ regularity theorem is however false, as witnessed by the following example.
Let $\delta>0$ be small and consider the $2$-dimensional surface $\Sigma$ in $\mathbb R^4 = \mathbb C^2$ given by
\begin{equation}\label{e:esempio_1}
\Sigma = \left\{(z,w)\in B_1 \times \mathbb C : \delta^2 z^3 = w^2\right\}\, . 
\end{equation}
A theorem of Federer (based on a computation of Wirtinger) guarantees that $\Sigma$ is an area minimizing oriented surface (without boundary in the cylinder $B_1\times \mathbb R^n$): the Federer-Wirtinger theorem is in fact valid for every holomorphic subvariety of $\mathbb C^n$.  By choosing $\delta$ arbitrarily small we can make \eqref{e:excess10} arbitrarily small. On the other hand there is no neighborhood of the origin in which $\Sigma$ can be described by disjoint $C^1$ graphs over $B_1 \times \{0\} \subset \mathbb C\times \{0\}$. 

Proving the Federer-Wirtinger theorem requires the introduction of the technology of Federer-Fleming integral currents and goes far beyond the scope of these notes (although the relevant idea is elementary; cf. for instance \cite{DL_CDM}). The important message is however that in higher codimension area minimizing oriented $m$-dimensional surfaces can have branching singularities of dimension $m-2$, whereas the latter singularities are not present in area minimizing oriented {\em hypersurfaces}. This phenomenon creates a wealth of extra difficulties for the regularity theory of area minimizing integral currents in codimension higher than $1$.

\subsection{Almgren's regularity theory and the ``center manifold''} In the seventies and early eighties Almgren wrote a celebrated long monograph,
see \cite{Alm}, dedicated to the regularity theory of area minimizing currents in higher codimension, where he was able to finally tackle the presence of branching singularities and prove an optimal dimension bound for them. This complicated theory was recently significantly simplified in a series of joint works by Emanuele Spadaro and the author, see \cites{DS1,DS2,DS3,DS4,DS5} and the survey articles 
\cites{DL_ICM,DL_Survey,DL_CDM}. The latter works have also sparked further research in the area, going beyond Almgren's theory and answering to some open questions in \cite{Alm}, cf. \cites{Spadaro,DFS,Hirsch1,Hirsch2,KW1,FMS,Hirsch3,LM,DSS1,DSS2,DSS3,DSS4,DL_BUMI,GS,Spolaor,DMSV,HSV,KW2,Stuvard1,Stuvard2,DDHM}. 

The most difficult part of Almgren's theory is the construction of what he calls ``center manifold''. In a nutshell, if we consider the example of Section \ref{ss:branching}, we can regard it as a ``two-sheeted'' cover of $B_1\times \mathbb C$. Although such two-sheeted cover has a singularity in $0$, the average of the two sheets is in fact precisely $B_1  \times \{0\}$, so it is a smooth graph over the base.

Almgren's center manifold is a powerful generalization of the latter observation: in an appropriate sense, when $\bE$ in \eqref{e:excess10} is small, the area minimizing surface $\Sigma$ is close to a multiple cover of the base and the average of the sheets enjoys better properties than the whole object. A nontechnical formulation is that it is possible to construct an efficient $C^{3,\beta}$ approximation of the average (with $\beta>0$ a small positive dimensional constant): namely, the $C^{3,\beta}$ norm of such approximation is bounded by $\mathbf{E}^{\frac{1}{2}}$ while the distance between the approximation and the average of the sheets is, at every scale, much smaller than the separation between the most distant sheets. In order to illustrare the subtlety of the above claim, consider the following example:
\[
\left\{(z,w)\in \mathbb C^2 : (w-z^2)^2 = z^{2019}\right\}\, .
\]
The latter surface, which has a branching singularity at the origin (as the surface in \eqref{e:esempio_1}), is a double cover of $B_1 \times \{0\}$. For every $z\neq 0$ with $|z|<<1$, we see two separate sheets above $B_{|z|} (2z)$. They are however extremely close: indeed their separation is smaller than $|z|^{1009}$. At that scale the center manifold must thus approximate the average, which is the map $z\mapsto z^2$, with a degree of precision which is much much higher than the size of the average itself. 

A precise statement of Almgren's theorem would require to go through the entire works \cites{DS1,DS2,DS3,DS4} and the interested reader is again referred to the survey articles \cite{DL_CDM,DL_Survey} for a first account of it. The latter reference explains also why it matters to have $C^{3,\beta}$ estimates. In these notes we focus on one ``striking'' corollary, remarked by Almgren himself in the introduction of \cite{Alm}. Under the assumption of Theorem \ref{t:de_giorgi} (where the ``separation between the most distant sheets'' is $0$, because there is only one sheet!) the center manifold must {\em coincide} with the very surface $\gr (u)$. This gives then the following

\begin{mythm}\label{t:almgren}
There are geometric constants $\beta, \varepsilon_0, C>0$ depending only on $m$ and $n$ with the following property.
Let $\Omega = B_1\subset \mathbb R^m$ and $u: B_1 \to \mathbb R^n$ be a Lipschitz map with $\Lip (u) \leq 1$ whose graph is area minimizing.
Assume
\begin{equation}\label{e:excess_bis}
E := \Vol^m (\gr (u, B_1)) - \omega_m < \varepsilon_0\, 
\end{equation}
and set $\sigma:= (2\sqrt{m})^{-1}$.
Then $u\in C^{3,\beta} (B_\sigma)$ and in fact $\|Du\|_{C^{2,\beta} (B_\sigma)} \leq C E^{\frac{1}{2}}$. 
\end{mythm}

At a first glance the latter statement is all but surprising. After all, De Giorgi's Theorem \ref{t:de_giorgi} allows to apply the classical Schauder estimates for ellyptic systems, hence we can give the estimate $\|Du\|_{C^k} \leq C (k,m,n) E^{\frac{1}{2}}$ for all $k\in \mathbb N$. However, the striking novelty is that Theorem \ref{t:almgren} can be proved without resorting to Schauder estimates and in fact without resorting to any PDE for the function $u$. Although we will use {\em PDE arguments}, they will be so elementary and robust that they can be used even in the general (i.e. multisheeted) situation. 

In \cite{Alm} the corollary above is observed in the introduction as a mere curiosity. For Spadaro and the author of these notes Almgren's remark was however a crucial starting point. After finding an elementary proof of the $C^{3,\beta}$ estimates under the assumptions of De Giorgi's $\varepsilon$-regularity theorem (cf. \cite{DS-cm}), we were able to give in \cite{DS4} a construction of the center manifold which seems much more efficient than Almgren's: while Almgren's proof is almost 500 pages long and occupies half of his monograph, the one of \cite{DS4} cuts the complexity by a factor 10 and its flexibility has proved to be very useful in other contexts (cf. \cites{DSS3,Spolaor,DDHM}). In these lecture notes we will follow essentially \cite{DS-cm} in the simplified setting of Lipschitz graphs to give a ``geometric'' proof of Theorem \ref{t:almgren}.

\section{Improved Lipschitz approximation}

In order to prove both theorems we will make use of a preliminary important estimate: under the small excess assumption, an area minimizing graph turns out to have a much smaller Lipschitz constant on a rather large subset of its domain. The same statement is still correct in the multisheeted situation and it is a fundamental step in Almgren's regularity theory, cf. \cite{DL_Survey}. The proof which we will give in these notes is a simplification of the one given in \cite{DS3} in the multisheeted situation.

\subsection{Spherical excess and scaling invariance.} Consider an oriented surface $\Sigma$ of dimension $m$ in $\mathbb R^{m+n}$ and for every $p\in \Sigma$ denote by $\vec{T}_p \Sigma$ the unit orienting $m$-vector of the tangent space $T_p \Sigma$. 

\begin{mydef} The spherical excess of $\Sigma$ in the ball $\bB_r (p)\subset \mathbb R^{m+n}$ with respect to a unit simple $m$-vector $\vec{\pi}$ is given by
\begin{equation}
\bE (\Sigma, \bB_r (p), \vec{\pi}) = \frac{1}{\omega_m r^m} \int_{\Sigma\cap \bB_r (p)} |\vec{T}_q \Sigma - \vec{\pi}|^2\, d\Vol^m (q)\, .
\end{equation}
The spherical excess of $\Sigma$ in $\bB_r (p)$ is defined as
\begin{equation}
\bE (\Sigma, \bB_r (p)) = \min_{\mbox{$\pi$ simple, $|\pi|=1$}} \bE (\Sigma, \bB_r (p), \vec{\pi})\, . 
\end{equation}
\end{mydef}

Before going on with our discussion, we want to introduce a very elementary yet powerful idea. If $u: B_r (x) \to \mathbb R^n$ is a Lipschitz map whose graph is area minimizing, then 
\[
u_r (y) := \frac{1}{r} (u (x+ry) - u(x))
\]
is a Lipschitz map such that
\begin{itemize}
\item $\Lip (u_r) = \Lip (u)$;
\item the graph of $u_r$ is area minimizing;
\item $\bE (\gr (u), \bB_\rho (x, u(x))) = \bE (\gr (u_r), \bB_{r^{-1} \rho} (0))$. 
\end{itemize}
In other words our problem has a natural invariance under scalings and translations which will be used through the notes to reduce the complexity of several proofs and to gain an intuition on the plausibility of the statements.

\subsection{Elementary remarks} Note the following obvious fact: if $\vec\pi$ is such that 
\[
\bE (\Sigma, \bB_\tau (p), \vec{\pi}) = \bE (\Sigma, \bB_\tau (p))\, ,
\]
then
\begin{align}
\bE(\Sigma, \bB_\sigma (q)) &\leq \bE (\Sigma, \bB_\sigma (q), \vec{\pi}) \leq \left(\frac{\tau}{\sigma}\right)^m \bE (\Sigma, \bB_\tau (p), \vec{\pi})\nonumber\\
&= \left(\frac{\tau}{\sigma}\right)^m \bE (\Sigma, \bB_\tau (p))\, \qquad\qquad\qquad \forall \bB_\sigma (q) \subset \bB_\tau (p) \, .\label{e:obvious}
\end{align}
Similar elementary considerations lead to the following proposition, whose proof is left as an exercise to the reader:

\begin{myprop}\label{p:compare1}
Let $u: \Omega \to \mathbb R^n$ be a map with $\Lip (u) \leq 2$, $p = (x,u(x))$ and $q= (y, u(y))$. Then there are 
geometric constants $C_1, C_2$ and $C_3$ such that
\begin{align}
& C_1^{-1} r^m \leq \Vol^m (\gr (u) \cap \bB_r (p)) \leq C_1 r^m\,\label{e:density_est}\\
& \qquad\qquad\quad \mbox{if } r < \dist (x, \partial \Omega)\, , \nonumber\\
& |\vec{\pi}_1 - \vec{\pi}_2|^2 \leq C_2 (\bE (\gr (u), \bB_{2r} (p), \vec{\pi}_1) + \bE (\gr (u), \bB_{\rho} (p), \vec{\pi}_2))  \label{e:r-2r}\\
&\qquad\qquad\quad\mbox{if } r \leq \rho \leq 2r < \dist (x, \partial \Omega)\nonumber
\end{align}
and
\begin{align}
&|\vec{\pi}_1 - \vec{\pi}_2|^2 \leq C_3 (\bE (\gr (u), \bB_r (p), \vec{\pi}_1) + \bE (\gr (u), \bB_r (q), \vec{\pi}_2)) \label{e:q-p}\\
&\qquad\qquad\quad\mbox{if } r= |p-q| < \min \{\dist (x, \partial \Omega), \dist (y, \partial \Omega)\}\, .\nonumber
\end{align}
\end{myprop}

\subsection{Comparing spherical and cylindrical excess.}\label{ss:sf-to-cil} We now want to compare the two slightly different notions of excess that we have given so far. Consider $p = (x,u(x))$. First of all, since $\bB_r (p) \subset \bC_r (x)$, we have the
obvious inequality
\[
\bE (\gr (u), \bB_r (p), \vec{\pi}) \leq 2 \omega_m^{-1} r^{-m} \bE (\gr (u), \bC_r (x), \vec{\pi})\, .
\] 
We next wish to show a sort of converse, under the assumption that the domain $\Omega$ of $u$ contains $B_1 (0)$.

By scaling invariance and translation invariance assume $x=0$, $u(0)=0$ and $r=1$. 
Under the assumption that $\bE (\gr (u), \bB_1, \vec{\pi}_0) < \varepsilon_1$ is rather small, fix a $\pi_1$ such that 
$\bE (\gr (u), \bB_1, \vec{\pi}_1) = \bE (\gr (u), \bB_1)$. Observe thus that, by Proposition \ref{p:compare1}, 
\begin{equation}\label{e:primo_tilt}
|\vec{\pi}_0- \vec{\pi}_1| \leq C \varepsilon_1^{\frac{1}{2}}\, .
\end{equation}
Denote by $\bp$ the orthogonal projection of $\mathbb R^{m+n}$ onto $\pi_1$. We claim that, for every $\eta>0$, by choosing
$\varepsilon_1$ sufficiently small, 
\begin{equation}\label{e:claim_eta}
\bp (\gr (u) \cap \bB_1) \supset \bB_{1-\eta} \cap \pi_1\, .
\end{equation}
If we denote by $\pi_1^\perp$ the orthogonal complement of $\pi_1$ and by $\bC'$ the cylinder 
\begin{equation}\label{e:cilindro_storto}
\bC' := (\bB_{1-\eta} \cap \pi_1)+\pi_1^\perp\, ,
\end{equation}
\eqref{e:claim_eta} obviously implies that
$\gr (u) \cap \bC'\subset \bB_1$ and allows us to establish the inequality $\bE (\gr (u), \bC'_{1-\eta}, \vec{\pi}_1) 
\leq \frac{\omega_m}{2} \bE (\gr (u), \bB_1)$.

We briefly sketch the proof of \eqref{e:claim_eta}. First we know from Lemma \ref{l:first_rotation} that $\gr (u)$ is the graph of some function $u': \Omega' \to \pi_1^\perp$ over some domain $\Omega' \subset \pi_1$. Moreover, by \eqref{e:primo_tilt}, Lemma \ref{l:first_rotation} gives us the estimate 
\[
\Lip (u') \leq 1 + C \varepsilon_1^{\frac{1}{2}}\leq 2
\] 
and, since $(0,0) \in \gr (u', \Omega')$, $u'(0)=0$. Thus 
$\gr (u)\cap \bB_1$ certainly contains $\gr (u', B_{1/4})$, provided $\varepsilon_1$ is smaller than a geometric constant. Let now $\rho$ be the maximal radius for which $B_\rho \subset \Omega'$ and $\gr (u', B_\rho )\subset \bB_1$.
First observe that
\begin{equation}\label{e:L2-bound}
\int_{B_\rho} |Du'|^2 \leq C \varepsilon_1\, ,
\end{equation}
because $|Du' (x')| \leq C |\vec{T}_{(x', u' (x'))} \gr (u') - \vec{\pi}_1|$. 
Next we can interpolate between \eqref{e:L2-bound} and $\|Du'\|_{L^\infty}\leq 2$ to easily conclude
\[
\|Du'\|_{L^{2m} (B_\rho)} \leq C \|Du'\|_{L^\infty (B_\rho)}^{1-\frac{1}{m}} \|Du'\|_{L^2 (B_\rho)}^{\frac{1}{m}}
\leq C \varepsilon_1^{\frac{1}{2m}}\, .
\]
Using Morrey's embedding and $u' (0) =0$ we thus get
\[
\|u'\|_{L^\infty (B_\rho)} \leq C \varepsilon_1^{\frac{1}{2m}}\, .
\]
However by the Lipschitz regularity of $u'$ and the maximality of $\rho$, there must be a point $x'$ with $|x'|= \rho$ such that either
$|(x', u' (x'))|=1$ or $\lim_{\sigma \to 1} ( \sigma x', u( \sigma x')) \not\in \gr  (u, \Omega)$.
In the second case, consider to have extended the function $u$ to the closure of $\Omega$. Since $\Omega$ contains $B_1$, we then must have that $|(x', u(x'))|\geq 1$, in particular we fall again in the first case. Thus $\rho^2 + |u' (x)'|^2 = |x'|^2 + |u' (x)'|^2 =1$. Inserting the estimate for $\|u'\|_{L^\infty}$, we get
\[
\rho^2 \geq 1 - C \varepsilon_1^{\frac{1}{m}}\, .
\]
For $\varepsilon_1$ sufficiently small the latter inequality guarantees $\rho \geq 1-\eta$ (in fact we can give an effective bound for how small $\varepsilon_1$ needs to be in terms of $\eta$, namely $\varepsilon_1\leq C (1-\eta)^{m}$ suffices; these type of bounds are indeed valid in general for suitable generalizations of surfaces which are stationary for the area functional). 

We summarize our discussion in the following lemma.

\begin{mylem}\label{l:sp-to-cyl}
For every $\eta>0$ there is $\bar\varepsilon >0$ with the following property. Let $u: B_1 \to \mathbb R^n$ be a Lipschitz map with $\Lip (u) \leq 1$ and $u(0)=0$ and assume that $\bE (\gr (u), \bB_1, \vec{\pi}_0) \leq \bar{\varepsilon}$. Let $\vec{\pi}_1$ be the unit $m$-vector such that $\bE (\gr (u) , \bB_1, \vec{\pi}_1) = \bE (\gr (u), \bB_1)$, let $\pi_1$ be the $m$-dimensional plane which is oriented by $\vec{\pi}_1$ and denote by $\bC$ the cylinder in \eqref{e:cilindro_storto}. Then:
\begin{itemize}
\item $\bC' \cap \gr (u)$ is the graph of a Lipschitz map 
\[
v: \bB_{1-\eta}\cap \pi_1\to \pi_1^\perp
\]
with $\Lip (v) \leq 2$,
\item $\gr (v) \subset \gr (u) \cap \bB_1$, 
\item and 
\begin{equation}
\bE (\gr (v), \bC') \leq \frac{\omega_m}{2} \bE (\gr (u), \bB_1)\, .
\end{equation}
\end{itemize} 
\end{mylem}

\subsection{Improved Lipschitz approximation.}\label{ss:Lipschitz} We now fix our attention on the map $v$ of the above lemma. Since the cylindrical excess $E$ controls $\int |Dv|^2$, the classical Chebyshev's inequality shows that
\[
|\left\{|Dv|\leq E^\gamma\right\}\leq C E^{1-2\gamma}\, .
\] 
for every positive exponent $\gamma$. More refined arguments from real harmonic analysis show a stronger result: there is a suitable set of measure no larger than $C E^{1-2\gamma}$ such that the restriction of $v$ to its complement has Lipschitz constant at most $E^\gamma$. For general functions this is the best we can do. However, the minimality assumption on $\gr (v)$ allows to give a much better bound, which is stated in Proposition \ref{p:lip_approx} below.

The proof of the proposition will introduce two important points: 
\begin{itemize}
\item The fact that the Dirichlet energy and the area integrand are comparable in areas where the tangents to the graph are almost horizontal;
\item A ``cut-and-paste'' idea to construct competitors for testing the minimality of the graph of $v$.
\end{itemize}
Both these two points will be exploited often in the rest of the notes: in the proof of the next proposition we will see
all the details, but at later stages we will be less precise and just refer to the arguments of this section.

\begin{myprop}\label{p:lip_approx}
There are $\bar \varepsilon>0$ and $\gamma>0$ geometric constants with the following property.
Let $v: B_r (x) \to \mathbb R^n$ be a Lipschitz function with $\Lip (v) \leq 2$ whose graph is area minimizing and such that
\[
E := r^{-m} \bE (\gr (v), \bC_r (x)) < \bar \varepsilon\, .
\]
If we set $\rho:= r (1-E^\gamma)$, then there is a set $K\subset B_\rho (x)$ such that
\begin{align}
|B_\rho (x)\setminus K| &\leq r^m E^{1+\gamma}\\
\Lip (v|_K) &\leq E^\gamma\, .
\end{align}
\end{myprop}

Before coming to the proof we introduce a useful terminology, which will be used often in the rest of the notes

\begin{mydef}\label{d:Lip}
Consider the set $K$ of Proposition \ref{p:lip_approx} and let $w$ be a Lipschitz extension of $v|_K$ to $B_\rho$ which does not increase its Lipschitz constant by more than a geometric factor\footnote{Indeed by Kirszbraun's Theorem we can require $w$ to have the {\em same} Lipschitz constant as $v|_K$. This is however a sophisticated theorem, whereas an extension which looses the geometric factor $\sqrt{n}$ can be constructed in an elementary way and suffices for our purposes.}. Although $w$ is not unique, we will call it the {\em $E^\gamma$-Lipschitz approximation of $v$}.
\end{mydef}

\begin{proof} First of all observe that we can allow a geometric constant $C$ in front of the estimates: then choosing a smaller exponent $\gamma$ and a sufficiently small $\bar\varepsilon$ we can eliminate the constant. Secondly we will focus on the proof of the following weaker statement: there is an $\omega>0$ (which depends only upon $m$ and $n$) and a set $L\subset B_{r/2} (x)$ satisfying
\begin{align}
|B_{r/2} (x)\setminus L|\leq C r^m E^{1+\omega}\label{e:claim_omega1}\\
\Lip (v|_L) \leq C E^\omega\, ,\label{e:claim_omega2}
\end{align}
for some exponent $\omega>0$. It can be in fact easily checked that the method of proof allows to pass from $r$ to $r (1-E^\mu)$ for a suitably small exponent $\mu$ and $\gamma$ can then be chosen to be $\min \, \{\omega, \mu\}$. 
Finally, by scaling and translating, without loss of generality we can prove \eqref{e:claim_omega1}-\eqref{e:claim_omega2} when $r=2$ and $x=0$. 

Summarizing, we are left with proving the following. Let $v: B_2 \to \mathbb R^n$ be a Lipschitz map with $\Lip (v)\leq 2$, whose graph is area minimizing and
\[
E := \bE (\gr (v), \bC_2) < \bar{\varepsilon}\, .
\]
We look for a set $K\subset B_1$ such that
\begin{align}
|B_1 \setminus K| &\leq C E^{1+\omega}\\
\Lip (v|_K) &\leq C E^\omega\, .
\end{align}
First of all observe that, by \eqref{e:formula_excess}, \eqref{e:area_formula} and elementary properties of the integrand in \eqref{e:area_formula},
\[
C^{-1} \int_{B_2} |Dv|^2 \leq E \leq C \int_{B_2} |Dv|^2\, 
\]
where $C$ is a geometric constant. Let $M |Dv|^2$ denote the usual maximal function
\[
M |Dv|^2 (y) := \sup_{r<1/2} r^{-m}\int_{B_r (y)} |Dv|^2\, .
\]
Recalling the classical weak $L^1$ estimate (see for instance \cite{Stein}), if we set 
\[
K:= \left\{y\in B_{3/2} : M |Dv|^2 \leq E^{2\lambda}\right\}\, ,
\] 
then\footnote{At this point the reader can easily check that in fact the estimate is valid for 
\[
K':= \left\{y\in B_{2-E^\mu} : M |Dv|^2\leq E^{2\lambda}\right\}
\]
in place of $K$, where $\mu$ is a suitable positive exponent depending on $m$ and $\lambda$. In this and similar considerations one can refine all the arguments to pass from the outer radius $2$ to an inner radius $2 - E^\mu$}
\begin{equation}\label{e:weak_L1}
|B_{3/2} \setminus K|\leq C E^{-2\lambda} \int_{B_2} |Dv|^2 \leq C E^{1-2\lambda}\, .
\end{equation}
Classical Sobolev space theory (see for instance \cite{EG}) implies that $\Lip (v|_K) \leq C E^\lambda$. We must therefore improve 
upon \eqref{e:weak_L1} using the minimality of $\gr (v)$. 

First of all we let $w$ be a Lipschitz extension of $v$ to $B_{3/2}$ with 
\[
\Lip (w) \leq \sqrt{n}\, \Lip (v|_K) \leq C E^{\lambda}\, .
\] 
We fix next a parameter $\vartheta>0$ and 
define $z = w \ast \varphi_{E^\vartheta}$, where $\varphi$ is a standard smooth mollifier. Observe that
\[
\int_{B_{5/4}} |w-z|^2 \leq C E^{2\vartheta} \int_{B_{3/2}} |Dw|^2 \leq C E^{1+2\vartheta}\, ,
\]
whereas
\begin{align*}
\|w-z\|_{C^0} &\leq C E^\vartheta\\
\int_{B_{5/4}} |Dw|^2 &\leq CE\, . 
\end{align*}
Using Fubini's theorem we choose a radius $\sigma\in ]9/8, 5/4[$ with the
property that $\Vol^{m-1} (\partial B_\sigma \setminus K) \leq C E^{1-2\lambda}$ and such that
\begin{align}
&\int_{\partial B_\sigma} (|Dv|^2 +|Dw|^2 + |Dz|^2) \leq C E\\
&\int_{\partial B_\sigma} |w-z|^2 \leq C E^{1+2\vartheta}\, .
\end{align}We fix next a second parameter $\kappa$, which we use to define the radii
\begin{align*}
r_1 &= \sigma - E^{\kappa}\\
r_2 &= \sigma - 2 E^{\kappa}\, ,
\end{align*}
whereas we set $r_0 = \sigma + E^\vartheta$.
It is not difficult to see that we can additionally require
\begin{equation}\label{e:small_layer}
\int_{B_{r_0}\setminus B_\sigma} |Dv|^2 \leq C E^{1+\vartheta}\, .
\end{equation}
The idea is now to define a new function $v'$ such that
\begin{itemize}
\item $v' = v$ on $\partial B_\sigma$;
\item $v' (x) = w (\sigma \frac{x}{r_1})$ on $\partial B_{r_1}$;
\item $v' (x) = z (\sigma \frac{x}{r_2})$ on $B_{r_2}$.
\end{itemize}
In the annuli $B_\sigma\setminus B_{r_1}$ and $B_{r_1}\setminus B_{r_2}$ we wish to define the function $v'$ ``interpolating'' between the values on the corresponding spheres and keeping the Dirichlet energy under control. It is not difficult to see that this can be done with a linear interpolation along the radii so to have the following estimate on the annalus $B_\sigma \setminus B_{r_1}$
\[
\int_{B_\sigma \setminus B_{r_1}} |Dv'|^2 \leq C E^\kappa \int_{\partial B_\sigma} (|Dw|^2 + |Dv|^2) 
+ C E^{-\kappa} \int_{\partial B_\sigma} |w-v|^2\, 
\]
and an analogous one (left to the reader as an exercise) in the annulus $B_{r_1}\setminus B_{r_2}$.

Now, recall that $\{w\neq v\} \subset K$, $|\partial B_\sigma \setminus K| \leq C E^{1-2\lambda}$ and both functions have Lipschitz constant no larger than $2$. It is then easy to see that 
\[
\|w-v\|_{C^0 (B_\sigma)} \leq C E^{(1-2\lambda)/(m-1)}\, .
\] 
We thefore achieve
\[
\int_{B_\sigma \setminus B_{r_1}} |Dv'|^2 \leq C E^{1+\kappa} + C E^{(1-2\lambda)(1+ 2 (m-1)^{-1}) - \kappa}\, .
\]
In particular, by choosing $\lambda$ and then $\kappa$ sufficiently small we conclude that
\begin{equation}\label{e:one_annulus}
\int_{B_\sigma \setminus B_{r_1}} |Dv'|^2 \leq C E^{1+\kappa}\, .
\end{equation}
Similarly, we can estimate
\[
\int_{B_{r_1}\setminus B_{r_2}} |Dv'|^2 \leq C E^{1+\kappa} + C E^{1+2\vartheta -\kappa}\, .
\]
Assuming thus that $\kappa<\vartheta$ we actually achieve
\begin{equation}\label{e:two_annuli}
\int_{B_\sigma\setminus B_{r_2}} |Dv'|^2 \leq C E^{1+\kappa}\, .
\end{equation}
We then leave to the reader to check that the Lipschitz constant of $v'$ is bounded by a constant, thus implying that
\begin{equation}\label{e:sommando_1}
\Vol^m (\gr (v'), B_{\sigma}\setminus B_{r_2}) \leq |B_{\sigma} \setminus B_{r_2}| + C E^{1+\kappa}\, .
\end{equation}

Next observe that
\begin{align*}
\int_{B_{r_2}} |Dv'|^2 &\leq \int_{B_\sigma} |Dz|^2\, . 
\end{align*}
Now we set $\ell = E^\vartheta$ and we write 
\[
|Dz| = |D w*\varphi_\ell| \leq |Dw|*\varphi_\ell \leq (|Dw|\mathbf{1}_K)*\varphi_\ell| + (|Dw| \mathbf{1}_{K^c})*\varphi_\ell\, .
\]
Observe that $r_0 = \sigma +\ell$ and thus
\begin{equation}\label{e:un_pezzo}
\int_{B_\sigma} |(|Dw|\mathbf{1}_K)*\varphi_\ell|^2 \leq \int_{B_{r_0}\cap K} |Dv|^2
\leq \int_{B_\sigma\cap K} |Dv|^2 + C E^{1+\vartheta}\, .
\end{equation}
The remaining part is estimated via
\begin{align}
\int_{B_\sigma} |(|Dw| \mathbf{1}_{K^c})*\varphi_\ell^2 &
\leq C E^{2\lambda} \|\mathbf{1}_{K^c}*\varphi_\ell\|_{L^2}^2 
\leq C E^{2\lambda}\|\mathbf{1}_{K^c}\|_{L^1}^2\|\varphi_\ell\|_{L^2}^2\nonumber\\
&\leq C E^{2\lambda} E^{2 (1-2\lambda)} E^{- m \vartheta} \leq C E^{1+\kappa}\, , \label{e:un_altro_pezzo}
\end{align}
provided $\lambda$ and $\vartheta$ are suitably chosen. Combining the last two estimates we easily achieve
\[
\int_{B_\sigma} |Dz|^2 \leq \int_{B_\sigma \cap K} |Dw|^2 + C E^{1+\kappa}
= \int_{B_\sigma \cap K} |Dv|^2 + C E^{1+\kappa}\, .
\]
Consider now that, since $\Lip (v'|_{B_{r_2}}) \leq \Lip (z) \leq \Lip (w) \leq C E^\lambda$, 
a simple Taylor expansion of the area functional\footnote{Recall that the integrand is 
\[
\left(1+ |A|^2 + \sum_{|\alpha \geq 2} (\det M_{\alpha\beta} (A))^2\right)^{\frac{1}{2}}\, .
\]
We then just need to observe that $|\det M_{\alpha\beta} (A)|\leq |A|^{|\alpha|}$ and  that $\sqrt{1+t} = 1 + \frac{t}{2} + O (t^2)$.
}
gives 
\begin{align}
\Vol^m (\gr (v'), B_{r_2}) &\leq |B_{r_2}| + \frac{1}{2} \int_{B_{r_2}} |Dv'|^2 + C E^{1+\kappa}\nonumber\\
&\leq |B_{r_2}| + \frac{1}{2} \int_{B_\sigma\cap K} |Dv|^2 + C E^{1+\kappa}
\label{e:sommando_2}
\end{align}
Summing \eqref{e:sommando_1} and \eqref{e:sommando_2} we thus get
\[
\Vol^m (\gr (v'), B_\sigma) \leq |B_\sigma| + \frac{1}{2} \int_{B_\sigma\cap K} |Dv|^2 + C E^{1+\kappa}\, .
\]
In particular, the minimality of $v$ implies that
\begin{equation}\label{e:comparison}
\Vol^m (\gr (v), B_\sigma) - |B_\sigma| \leq \frac{1}{2} \int_{B_\sigma\cap K} |Dv|^2 + C E^{1+\kappa}\, .
\end{equation}
Next, write
\begin{align*}
\Vol^m (\gr (v), B_\sigma) &= \Vol^m (\gr (v), B_\sigma \cap K) + \Vol^m (\gr (v), B_\sigma \setminus K)\\
&\geq \Vol^m (\gr (v), B_\sigma \cap K) + |B_\sigma \setminus K|\, .
\end{align*}
Recall however that $|Dv|\leq C E^\lambda$ on $K$. Hence we can use again the Taylor expansion of the area integrand 
and conclude
\[
\Vol^m (\gr (v), B_\sigma \cap K) \geq |B_\sigma \cap K| + \frac{1}{2} \int_{B_{\sigma} \cap K} |Dv|^2 - C E^{1+\kappa}\, ,
\]
whereas on $B_\sigma \setminus K$ we use the crude estimate
\[
C^{-1} \int_{B_\sigma \setminus K} |Dv|^2 \leq \Vol^m (\gr (v), B_{\sigma \setminus K}) - |B_\sigma \setminus K|\, .
\]
Summarizing
\begin{equation}\label{e:comparison2}
\Vol^m (\gr (v), B_\sigma) - |B_\sigma| \geq C^{-1} \int_{B_\sigma \setminus K} |Dv|^2 + \frac{1}{2} \int_{B_{\sigma} \cap K} |Dv|^2 - C E^{1+\kappa}\, .
\end{equation}
Combining \eqref{e:comparison} and \eqref{e:comparison2}, we get 
\[
\int_{B_\sigma\setminus K} |Dv|^2 \leq C E^{1+\kappa}\, .
\]
$K$ was given as $K:= \{M|Dv|^2 \geq E^{2\lambda}\}\cap B_{3/2}$. We have thus achieved the following: there are constants $\kappa>0$ and $\lambda_0>0$ such that, if $\lambda \leq \lambda_0$, then the inequality
\[
\int_{B_{9/8} \cap \{M |Dv|^2 \geq E^{2\lambda}\}} |Dv|^2 \leq C E^{1+\kappa}\, 
\] 
holds, provided $E$ is sufficiently small.

Set now $\omega := \min\{ \frac{\kappa}{3}, \frac{\lambda_0}{2}\}$.
We define
\[
m |Dv|^2 (y) = \sup_{r<\frac{1}{8}} \frac{1}{r^m} \int_{B_r (y)} |Dv|^2\, .
\]
and $L := \{y\in B_1 : m |Dv|^2 \leq E^{2\omega}\}$. We next recall the more precise form of the weak $L^1$ estimates for maximal functions, that is
\[
|B_1\setminus L| \leq C E^{-2\omega} \int_{B_{9/8} \cap \{m |Dv|^2 \geq C^{-1} E^{2\omega}\}} |Dv|^2\, .
\]
where $C$ is a geometric constant.
It is easy to see that, for $E$ small, if $m|Dv|^2 \geq C^{-1} E^{2\omega}$, then $M |Dv|^2 \geq C^{-1} E^{2\omega}
\geq E^{2\lambda_0}$. Hence we conclude
\[
|B_1 \setminus L|\leq C E^{1+\kappa-2\omega} \leq C E^{1+\omega}\, .
\]
Since $\Lip (v|_L) \leq C E^\omega$, this completes the proof. 
\end{proof}

\section{De Giorgi's excess decay and the proof of Theorem \ref{t:de_giorgi}}

We now examine Theorem \ref{t:de_giorgi}. The key idea is that, under the assumptions of the theorem, the {\em spherical excess} decays geometrically at smaller scales: such decay is an effect of the almost harmonicity of $u$, which in turn is again a consequence the Taylor expansion of the area integrand, computed on the improved Lipschitz approximation. 

\subsection{Excess decay.} 

De Giorgi's excess decay is the following proposition (which by scaling and translating we could have equivalently
stated with $x=0$ and $r=1$). 

\begin{myprop}\label{p:excess_decay}
For every $0<\alpha<1$ there is a geometric constant $\varepsilon_1$ depending only on $\alpha$, $m$ and $n$ with the following property. Assume $u: B_r (x) \to \mathbb R^n$ is a Lipschitz map with $\Lip (u) \leq 1$ whose graph is area minimizing and let $p= (x, u(x))$. If 
\begin{equation}\label{e:excess_small_2}
\bE (\gr (u), \bB_r (p))  < \bE (\gr (u), \bB_r (p), \vec{\pi}_0) < \varepsilon_1\, .
\end{equation}
Then
\begin{equation}\label{e:excess_decay_2}
\bE (\gr (u), \bB_{r/2} (p)) \leq \frac{1}{2^{2\alpha}} \bE (\gr (u), \bB_r (p))\, .
\end{equation}
\end{myprop}

The next idea is that Proposition \ref{p:excess_decay} can be iterated on ``dyadic radii'' and combined with \eqref{e:obvious} to conclude

\begin{mycor}\label{c:excess_decay}
For every $0<\alpha<1$ there is a geometric constant $\varepsilon_2$ depending only on $\alpha$, $m$ and $n$ with the following property. Assume $u: B_r (x) \to \mathbb R^n$ is as in Proposition \ref{p:excess_decay} with $\varepsilon_2$ substituting $\varepsilon_1$. Then
\begin{equation}\label{e:excess_decay}
\bE (\gr (u), \bB_{\rho} (q)) \leq 2^{2m+2\alpha} \left(\frac{\rho}{r}\right)^{2\alpha} \bE (\gr (u), \bB_r (p))\qquad \forall \bB_\rho (q) \subset \bB_{r/2} (p)\, .
\end{equation}
\end{mycor}

\begin{proof}[Proof of Corollary \ref{c:excess_decay}] By scaling and translating we can assume that $q=(0,0)$ and $r=2$. If $\varepsilon_2 < \varepsilon_1$, we then have
\[
\bE(\gr (u), \bB_{1/2} ) \leq 2^{-2\alpha} \bE (\gr (u), \bB_1) < 2^{-2\alpha} 2^m \bE (\gr (u), \bB_2 (p)) < 2^{m-2\alpha} \varepsilon_2\, .
\]
Next let $\pi_2$ and $\pi_1$ be such that 
\begin{align*}
&\bE (\gr (u), \bB_{1/2}, \vec{\pi}_2) = \bE (\gr (u), \bB_{1/2})\\
&\bE (\gr (u), \bB_1, \vec{\pi}_1) = \bE (\gr (u), \bB_1)\, .
\end{align*}
Applying Proposition \ref{p:compare1} we easily get 
\[
|\vec{\pi}_2-\vec{\pi_1}|^2 \leq \bar C \bE (\gr (u), \bB_1) \, 
\]
and
\[
|\vec{\pi}_1 - \vec{\pi}_0|^2\leq \bar C \bE (\gr (u), \bB_1, \vec{\pi}_0) \, ,
\]
where $\bar C$ is a geometric constant. In particular, choosing $\varepsilon_2$ much smaller than $\varepsilon_1$ we can use Proposition
\ref{p:compare1} to show that
\[
\bE (\gr (u), \bB_{1/2}, \vec{\pi}_0) \leq \bar C |\vec{\pi}_2-\vec{\pi}_0|^2 + 2 \bE (\gr (u), \bB_{1/2}) < \varepsilon_1\, 
\]
and we can apply once again Proposition \ref{p:excess_decay} to estimate
\[
\bE (\gr (u), \bB_{1/4}) \leq 2^{-2\alpha} \bE (\gr (u), \bB_{1/2}) \leq 2^{-4\alpha} \bE (\gr (u), \bB_1)\, .
\]
Assume now inductively that you are in the position of applying Proposition \ref{p:excess_decay} on all radii $r = 2^{-j}$ for $j=0, \ldots, k$. 
We then get
\[
\bE (\gr (u), \bB_{2^{-k}})\leq 2^{-2k\alpha} \bE (\gr (u), \bB_1)
\]
and, if $\bE (\gr (u), \bB_{2^{-k}}, \vec{\pi}_{k+1}) = \bE (\gr (u), \bB_{2^{-k}})$, we conclude
\[
|\vec{\pi}_{k+1} - \vec{\pi}_1|\leq \sum_{j=1}^k |\vec{\pi}_{k+1} - \vec{\pi}_1|
\leq \bar{C}^{\frac{1}{2}} \sum_{j=1}^\infty 2^{-2(j-1)\alpha} \bE (\gr (u) \bB_1)^{\frac{1}{2}}
\leq C(\alpha) \varepsilon_2^{\frac{1}{2}}\, .
\]
Hence, if $\varepsilon_2$ is sufficiently small compared to $\varepsilon_1$ (by a factor which depends on $\alpha$ {\em but not} on $k$),
we can argue as above and conclude that $\bE (\gr (u), \bB_{2^{-k}}, \vec{\pi}_0) < \varepsilon_1$. We are thus in the position of
applying Proposition \ref{p:excess_decay} even with $r=2^{-k}$. 

This proves inductively that $\bE (\gr (u), \bB_{2^{-k}}) \leq 2^{-2k\alpha} \bE (\gr (u), \bB_1)$. Observe now that, given any $\rho < 1$,
if we let $k = \lfloor - \log_2 \rho \rfloor$, then $2^{-k-1} \leq \rho \leq 2^{-k}$ and thus
\begin{align*}
\bE (\gr (u), \bB_\rho) &\leq 2^m \bE (\gr (u), \bB_{2^{-k}}) \leq 2^m 2^{-2k \alpha} \bE (\gr (u), \bB_1)\\
&\leq 2^{2m+2\alpha} \rho^{2\alpha} \bE (\gr (u), \bB_2 (p))\, .
\end{align*}
\end{proof}

\subsection{Proof of Theorem \ref{t:de_giorgi}} We now see how Corollary \ref{c:excess_decay} leads
quickly to Theorem \ref{t:de_giorgi}. 
First, we exploit the graphicality to compare the spherical excess to the square mean oscillation of $Du$. In the rest of the note we will
use the notation $(Du)_{x,\rho}$ to denote the average
\[
(Du)_{x,\rho} = \frac{1}{\omega_m \rho^m} \int_{B_\rho (x)} Du (y)\, dy\, . 
\]

\begin{myprop}\label{p:compare2}
Let $u: \Omega \to \mathbb R^n$ be a map with $\Lip (u) \leq 1$ and let $p = (x,u(x))$. There is a geometric constant $C>1$ such that,
if $B_{Cr} (x) \subset \Omega$, then
\begin{equation}\label{e:est}
\int_{B_r (x)} |Du(y) - (Du)_{x,r}|^2 dy \leq C r^m \bE (\gr (u), \bB_{4r} (p))\, .
\end{equation}
\end{myprop}
\begin{proof}
First of all observe that, by the Lipschitz bound on $u$, $\gr (u, B_r (x)) \subset \bB_{4r} (p)$. Next, let $\bE (\gr (u), \bB_{4r} (p), \vec\pi) 
= \bE (\gr (u), \bB_{4r} (p))$. Observe that, again by the Lipschitz bound, the plane oriented by $\vec{\pi}$ is the graph of a linear map
\[
\mathbb R^m \ni x \mapsto A x \in \mathbb R^n
\] 
with $|A|\leq C_0$ for some geometric constant $C_0$. An elementary geometric computation then gives
that $|A - Du (y)|\leq C |\vec{\pi} - \vec{T}_{(y, u(y))} \gr (u)|$. Using again the Lipschitz bound and the area formula we conclude therefore
\begin{align*}
\int_{B_r (x)} |A - Du(y)|^2\, dy &\leq C \int_{\bB_{4r} (p)\cap \gr (u)} |\vec{\pi} - \vec{T}_q \gr (u)|^2 d\Vol^m (q)\\
&\leq C r^m \bE (\gr (u), \bB_{4r} (p))\, .
\end{align*}
To achieve \eqref{e:est} recall then that
\[
\int_{B_r (x)} |Du(y) - (Du)_{x,r}|^2 = \min_A \int_{B_r (x)} |A - Du(y)|^2\, .
\]
\end{proof}

\begin{proof}[Proof of Theorem \ref{t:de_giorgi}]
From Proposition \ref{p:cil-sfer} and \eqref{e:obvious} we easily infer that
\[
\bE (\gr (u), \bB_{1/2} (p)) \leq 2^m \varepsilon_0 \qquad \mbox{for all $p =(x, u(x))$ with $x\in B_{1/2}$.}
\]
If $2^m \varepsilon_0$ is smaller than the threshold $\varepsilon_2$ in Corollary \ref{c:excess_decay} we can then combine it
with Proposition \ref{p:compare2} to conclude
\begin{equation}\label{e:morrey}
\int_{B_r (x)} |Du(y) - (Du)_{x,r}|^2 dy \leq C r^{m+2\alpha} E \qquad \forall x\in B_{1/2}, \forall r<\frac{1}{8}\, .
\end{equation}
It is a well-known lemma, due to Morrey, that \eqref{e:morrey} implies $\|Du\|_{C^\alpha}\leq C E^{\frac{1}{2}}$. We briefly sketch the proof.
First observe that, for $r< \frac{1}{16}$, 
\begin{align*}
& |(Du)_{x,r} - (Du)_{x,2r}|^2\\ 
\leq\; & C r^{-m} \left(\int_{B_r (x)} |Du(y) - (Du)_{x,r}|^2 + \int_{B_r (x)} |Du(y) - (Du)_{x,2r}|^2\right)\\
\leq\; &  C r^{-m} \left(\int_{B_r (x)} |Du(y) - (Du)_{x,r}|^2 + \int_{B_{2r} (x)} |Du(y) - (Du)_{x,2r}|^2\right) \leq C r^{2\alpha} E\, .
\end{align*}
In particular iterating on dyadic radii we easily get the estimate
\[
|(Du)_{x, 2^{-k}} - (Du)_{x, 2^{-j}}| \leq C E^{\frac{1}{2}} \sum_{j \leq i \leq k-1} 2^{-i\alpha} E \leq C 2^{-j \alpha} E^{\frac{1}{2}}\, \qquad \forall k\geq j \geq 3\, ,
\]
from which in turn we infer
\[
|(Du)_{x,r} - Du_{x,\rho}|\leq C \rho^{\alpha} E^{\frac{1}{2}} \qquad \forall r \leq \rho \leq \frac{1}{8}\, .
\]
If $x$ is a Lebesgue point for $Du$ we then conclude
\begin{equation}\label{e:Lebesgue}
|Du (x) - Du_{x, \rho}|\leq C \rho^{\alpha} E^{\frac{1}{2}} \qquad \forall \rho \leq \frac{1}{8}\, . 
\end{equation}
Fix now two points $x,y\in B_{1/2} (0)$ with $2 \rho:= |x-y|\leq \frac{1}{8}$ and let $z=\frac{x+y}{2}$. Observe that $B_\rho (z) \subset B_{2\rho} (x) \cap B_{2\rho} (y)$ to infer
\begin{align*}
& |(Du)_{x, 2\rho} - (Du)_{y, 2\rho}|^2\\
\leq\; & C \rho^{-m} \left(\int_{B_\rho (z)} |Du - (Du)_{x,2\rho}|^2 + \int_{B_\rho (z)} |Du - (Du)_{y, 2\rho}|^2\right)\\
\leq\; & C \rho^{-m} \left(\int_{B_{2\rho} (x)} |Du - (Du)_{x,2\rho}|^2 + \int_{B_{2\rho} (y)} |Du - (Du)_{y, 2\rho}|^2\right)
\leq C \rho^{2\alpha} E \, .
\end{align*}
Combining the latter estimate with \eqref{e:Lebesgue} we conclude
\[
|Du (x)- Du(y)|\leq C |x-y|^\alpha E^{\frac{1}{2}}
\]
whenever $x, y\in B_{1/2} (0)$ are Lebesgue points for $Du$ with $|x-y|\leq \frac{1}{8}$. The conclusion of the theorem follows then from
simple calculus considerations. 
\end{proof}

\subsection{Proof of the excess decay: harmonic blow-up.} In the rest of the chapter we focus on the proof of  Proposition \ref{p:excess_decay}. The considerations of Section \ref{ss:sf-to-cil} reduce it to the following decay of the ``cylindrical excess''. The simple details of such reduction are left to the reader.

\begin{myprop}\label{p:excess_decay_3}
For every $0< \alpha <1$ there are $\varepsilon_3>0$ and $\eta>0$ with the following property.
Let $v: B_{1-\eta} (0) \to \mathbb R^n$ be a Lipschitz function with $\Lip (v) \leq 2$ whose graph is area minimizing and such that
\[
\bE (\gr (v), \bC_{1-\eta}) < \varepsilon_3\, .
\]
Then there is a unit $m$-vector $\vec{\pi}$ such that
\begin{equation}\label{e:excess_decay_cil}
\bE (\gr (v), \bC_{1/2}, \vec{\pi}) \leq 2^{-m-2\alpha} \bE (\gr (v), \bC_{1-\eta})\, .
\end{equation}
\end{myprop}

Using the $E^\gamma$-Lipschitz approximation $w$ of $v$ and a Taylor expansion, we easily see that the the Dirichlet energy of $w$ and the excess of $u$ are pretty close, more precisely
\[
\frac{1}{2} \int_{B_{r (1-E^\gamma)}} |Dw|^2 = E + O( E^{1+\gamma})
\]
We next use the above estimate to show

\begin{myprop}\label{p:harm_bu}
Let $v_k : B_r (x)\to \mathbb R^n$ be a sequence of Lipschitz functions with $\Lip (v) \leq 2$ whose graphs are area minimizing and such that
\[
E_k := r^{-m} \bE (\gr (v_k), \bC_r) \downarrow 0\, .
\]
Consider the rescaled functions 
\[
f_k := \frac{v_k - (v_k)_{0, r (1-E^\gamma)}}{E_k^{\frac{1}{2}}}\, .
\]
Then $v_k$ converges, up to subsequences, strongly in $W^{1,2}_{loc} (B_r)$ to an harmonic function.
\end{myprop}

\begin{proof} Without loss of generality let $x=0$ and $r=1$ and assume $(v_k)_{0, 1-E^\gamma} = 0$. Note first that up to subsequences we can assume the existence of a weak limit $f\in W^{1,2} (B_1)$. Apply the Lipschitz approximation Proposition \ref{p:lip_approx} and let $K_k$ be the corresponding ``good sets''. Moreover let $w_k$ be a $E_k^\gamma$-Lipschitz extension of $v_k|_{K_k}$ to $B_{\rho_k} = B_{(1- E_k^\gamma)}$.
If we denote by $f'_k$ be the corresponding normalizations $E_k^{-\frac{1}{2}}w_k$, we still conclude that $f'_k$ converges to $f$. 
Assume by contradiction that for some radius $\rho<1$, the limit of the $f'_k$ is weak in the $W^{1,2}$ topology. We then must have
\[
\int_{B_\rho} |Df|^2 < \liminf_k \int_{B_\rho} |Df'_k|^2\, .
\]
Now, by using the cut-and-paste argument of Section \ref{ss:Lipschitz} and the Taylor expansion of the area functional, we would like to use the graph of $g_k = E_k^{\frac{1}{2}}f $ as a competitor for $\gr (v_k)$, violating the minimality of $\gr (v_k)$. We want to achieve this task by first pasting $v_k$ with $w_k$ over
an apprioprate annulus and then $g_k$ with $w_k$ in a second, slightly smaller, annulus, similarly to what was done in the proof of Proposition \ref{p:lip_approx} . Note that we have at our disposal the two crucial estimates which were used in the cut-and-paste argument:
\[
\int_{B_\rho} |Dg_k|^2 = O( E_k)
\] 
and
\[
\int_{B_\rho} |g_k - w_k|^2 = o (E_k)\, ,
\]
However, one important issue is that we do not know that $\Lip (g_k) \to 0$, which would be crucial to compare the  $\Vol^m (\gr (g_k))$ to the Dirichlet energy of $g_k$. In order to come around this issue fix a sequence of Lipschitz functions $h_j$ converging strongly in $W^{1,2}$ to $f$. A suitable diagonal sequence $E_k^{\frac{1}{2}} h_{j(k)}$ will have at the same time Lipschitz constants which converge to $0$ and will satisfy the estimates
needed to use the cut-and-paste argument.

The harmonicity of $f$ is proved in a similar way: if there is a competitor $h$ for $f$ in some $B_\sigma \subset\subset B_1$ with less Dirichlet energy, we fix
an intermediate radius $\rho$ between $\sigma$ and $1$ and we then run the argument above with $f$ replaced by the function
\[
f' (x) = \left\{
\begin{array}{ll}
f (x) \quad &\mbox{if $|x|\geq \sigma$}\\
h(x) \quad &\mbox{if $|x|\leq \sigma$}
\end{array}\right.
\]
The cut and paste argument will then be run in annuli contained in $B_\rho\setminus B_\sigma$. 
\end{proof}

\subsection{Cylindrical excess decay} We are now ready to complete the proof of Proposition \ref{p:excess_decay_3}.
The first ingredient is the following estimate for harmonic functions:

\begin{mylem}\label{l:harmonic}
Consider $h: B_r (x) \to \mathbb R^n$ harmonic and let $\rho< r$. Then
\begin{equation}
\int_{B_\rho (x)} |Dh - (Dh)_{x, \rho}|^2 \leq C \left(\frac{\rho}{r}\right)^{m+2} \int_{B_r (x)} |Dh|^2\, .
\end{equation}
\end{mylem}
\begin{proof} The proof is left to the reader: reduce it by scaling to the case $r=1$, 
use the decomposition of $h|_{\partial B_1}$ in spherical harmonics (see for instance \cite{SW}) and the mean-value theorem for harmonic functions.
\end{proof}

\begin{mylem}\label{e:Taylor_2}
Let $w: B_{1/2} \to \mathbb R^n$ be a Lipschitz function with $\Lip (w) \leq 2$. Set $A = (Dw)_{0, 1/2}$, consider the linear map $x\mapsto A x$ and
let $\vec{\pi}$ be the unit vector orienting its graph according to our definitions. We then have
\[
\bE (\gr (w), \bC_{1/2}, \vec{\pi}) \leq \frac{1}{2} \int_{B_{1/2}} |Dw- (Dw)_{0,1/2}|^2 + C \Lip (w) \int_{B_{1/2}} |Dw|^2\, .
\] 
\end{mylem}
\begin{proof} The proof is left to the reader:
it is a simple linear algebra computation combined with a classical Taylor expansion.
\end{proof}

\begin{proof}[Proof of Proposition \ref{p:excess_decay}] First let $w$ be the $E^\gamma$-Lipschitz approximation of the map $v$ (see Proposition
\ref{p:lip_approx} and Definition \ref{d:Lip}) and observe that, provided $E$ is sufficiently small, $w$ is defined on $B_{1-2\eta}$. Let $\vec{\pi}$ be the unit $m$-vector 
orienting the graph of the linear map $x\mapsto A x$ with $A= (Dw)_{0, 1/2}$. By Proposition \ref{p:lip_approx} and Lemma \ref{e:Taylor_2} we then have
\begin{align*}
\bE (\gr (v), \bC_{1/2}, \vec{\pi}) &\leq \bE (\gr (w), \bC_{1/2}, \vec{\pi}) + C E^{1+\gamma}\\
 &\leq \frac{1}{2} \int_{B_{1/2}} |Dw- (Dw)_{0,1/2}|^2 + C E^{1+\gamma}\, ,
\end{align*}
and 
\[
\frac{1}{2} \int_{B_{1-2\eta}} |Dw|^2 \leq E + C E ^{1+\gamma}\, .
\]
Use now Proposition \ref{p:harm_bu} to infer the existence of a harmonic function $h$ such that
$\| w - h\|^2_{W^{1,2} (B_{1-2\eta})} = o (E)$. In particular we can estimate
\[
\bE (\gr (v), \bC_{1/2}, \vec{\pi}) \leq \frac{1}{2} \int_{B_{1/2}} |Dh - (Dh)_{0,1/2}|^2 + o (E)
\]
and
\[
\frac{1}{2} \int_{B_{1-2\eta}} |Dh|^2 \leq E + o (E)\, .
\]
Using Lemma \ref{l:harmonic} we then infer
\begin{align*}
\bE (\gr (v), \bC_{1/2}, \vec{\pi}) &\leq \left(\frac{1-2\eta}{2}\right)^{m+2} E + o (E)\nonumber\\
&\leq \left(\left(\frac{1-2\eta}{2}\right) ^{m+2} + o(1)\right) \bE (\gr (v), \bC_{1-\eta})\, ,
\end{align*}
which is enough to complete the proof. 
\end{proof}

\section{Center manifold algorithm}

We next turn our attention to Theorem \ref{t:almgren}. As already mentioned, the claim could be easily proved by
first using De Giorgi's theorem to show that $u\in C^{1,\alpha}$, then deriving the Euler Lagrange equation for $u$ as
a minimizer of the area integrand and hence appealing to the Schauder estimates for elliptic systems. We instead decide
to ignore Schauder's estimates: we will introduce an efficient approximation algorithm producing a sequence of regularizations of $u$ which converges uniformly and for which we have uniform $C^{3,\beta}$ estimates for some positive $\beta$. In doing so we will even ignore the fact that $u\in C^{1,\alpha}$ and only use some of the corollaries of De Giorgi's excess decay.

\subsection{The grid and the $\pi_L$-approximations.} We start by considering the cube
\[
[-\sigma, \sigma]^m \subset B_{1/2} \subset B_1
\] 
and subdividing it in $2^{mk}$ {\em closed} cubes using a regular grid. We require that $k\geq N_0$, where $N_0$ will be specified in a moment. We will denote by $\ell (L)$ the sidelength of each cube of the grid and by $x_L$ its center. 

We denote by $p_L$ the point $p_L = (x_L, u (x_L))$ and we then let $\bB_L$ be the ball $\bB_{32 M_0 \ell (L)} (p_L)$, where $M_0$ is another sufficiently large geometric constant. Indeed $M_0 = \sqrt{n}$ suffices and this choice also dictates the choice of $N_0$: we want to guarantee that each $\bB_L$ is contained in the cylinder $\bC_1$ and thus we just need $32 \sqrt{n} \sigma 2^{-N_0} <1$. 

Next recall that, by the De Giorgi's excess decay, if we fix any $\delta>0$ we can assume
\begin{equation}\label{e:ex_decay_100}
\bE (\gr (u), \bB_L) \leq C \ell (L)^{2-2\delta} E\, 
\end{equation}
where $C = C(m,n,M_0, N_0, \delta)$. Let now $\pi_L$ be an oriented $m$-dimensional plane which optimizes the spherical excess\footnote{We do not discuss here whether such optimizer is unique, since it is irrelevant for the rest of the proof: if there is more than one optimal plane, we just fix an arbitrary choice.} in $\bB_L$, namely such that
\[
\bE (\gr (u), \bB_L, \vec{\pi}_L) = \bE (\gr (u), \bB_L)\, .
\]
Recalling the estimates of the previous chapter we get\footnote{Again the constants of the next estimates depend on $m,n,M_0, N_0$ and $\delta$: in the rest of the notes the constants will depend on these parameters unless we explicitely mention their dependence.}
\begin{equation}\label{e:quasi_orizzontale}
|\vec{\pi}_L - \vec{\pi}_0| \leq C E^{\frac{1}{2}}
\end{equation}
and in particular
\begin{equation}\label{e:eccesso_orizzontale}
\bE (\gr (u), \bB_L, \vec{\pi}_0) \leq C E\, .
\end{equation}
Since we will need it often, we now introduce a special notation to deal with tilted disks and cylinders. First of all we set
$B_r (p, \pi):= \bB_r (p) \cap (p+\pi)$ and hence we define 
\[
\bC_r (p, \pi) := B_r (p, \pi) + \pi^\perp\, ,
\]
where $\pi^\perp$ denotes the $n$-dimensional plane perpendicular to $\pi$. 

Applying Lemma \ref{l:sp-to-cyl} we conclude that $\bC_{16 M_0 \ell (L)} (p_L, \pi_L) \cap \gr (u)$ is in fact the graph
of a Lipschitz map $v_L : B_{16 M_0 \ell (L)} (p_L, \pi_L) \to \pi_L^\perp$ and we set
\[
E (L) := \ell(L)^{-m} \bE (\gr (u), \bC_{16 M_0 \ell (L)} (p_L, \pi_L))\, .
\]
Observe that the Lipschitz constant of $v_L$ is bounded by 
\[
\Lip (v_L) \leq \Lip (u) + C |\vec{\pi}_L- \vec{\pi}_0| \leq 1+ C E^{\frac{1}{2}}\, .
\] 
In particular, if $E$ is sufficiently small, 
\[
\bC_{16 M_0 \ell (L)} (p_L, \pi_L) \cap \gr (u) \subset \bB_{32 M_0 \ell (L)} (p_L, \pi_L)
\] 
and 
so
$E(L) \leq C E \ell (L)^{2-2\delta}$: we can thus apply 
Proposition \ref{p:lip_approx}.
We then denote by $f_L$ the $E(L)^\gamma$-Lipschitz approximation of $v_L$ in $\bC_{8 M_0 \ell (L)} (p_L, \pi_L)$. 
$f_L$ will be called the {\em $\pi_L$-approximation}. Moreover, recall that the functions $f_L$ and $v_L$ coincide on a large set,
more precisely
\begin{align}
|\{v_L\neq f_L\}\cap B_{8M_0 \ell (L)} (p_L, \pi_L)| &\leq C \ell (L)^m E (L)^{1+\gamma}\nonumber\\
 &\leq C E^{1+\gamma} \ell (L)^{m+ (2-2\delta)(1+\gamma)}\, .
\label{e:alessio_richiama_1}
\end{align}
A simple, yet useful, consequence of the latter estimate and the Lipschitz bounds on the two functions is then\footnote{$A\Delta B$ denotes the symmetric difference of the two sets $A$ and $B$, namely $A\Delta\, B = A \setminus B \cup B\setminus A$.}
\begin{equation}\label{e:alessio_richiama_2}
{\rm Vol}^m ((\gr (v_L) \Delta\, \gr (f_L)) \cap \bC_{8 M_0 \ell (L)} (p_L, \pi_L)) 
\leq C E^{1+\gamma} \ell (L)^{m+ (2-2\delta)(1+\gamma)}\, .
\end{equation}

\subsection{Interpolating functions and glued interpolations} Consider now a standard smooth function
$\varphi \in C^\infty_c (B_1)$ with $\int \varphi =1$ and let $\varphi_r$ be the corresponding family of mollifiers. We then set
\[
z_L := f_L \ast \varphi_{\ell (L)}\, .
\]
$z_L$ will be called the {\em tilted interpolating function} relative to the cube $L$.
We set conventionally $B_{4 M_0 \ell (L)} (p_L, \pi_L)$ to be the domain of definition of $z_L$. Clearly 
\[
\Lip (z_L) 
\leq C E(L)^\gamma \leq C E^\gamma \ell (L)^{(2-2\delta) \gamma}\, .
\]
Observe therefore
that we can use Lemma \ref{l:first_rotation} to infer the existence of a map $g_L: B_{2M_0 \ell (L)} (x_L, \pi_0) \to \pi_0^\perp$ such that
\[
\gr (z_L) \cap \bC_{2M_0 \ell (L)} (p_L, \pi_0) = \gr (g_L)\, .
\]
$g_L$ will be called the {\em interpolating function} relative to the cube $L$. 

Observe that the domain of $g_L$ contains the open cube $L'$ which is concentric to $L$ and has twice its side-length. 
Consider now a bump function $\vartheta\in C^\infty_c (]-\frac{9}{8}, \frac{9}{8}[^m)$ which is identically $1$ on the cube $[-1,1]^m$.
We then let 
\[
\vartheta_L (x) := \vartheta \left( \frac{2 (x-x_L)}{\ell (L)}\right)\, .
\]
Obviosuly $\vartheta_L$ is identically equal to $1$ on $L$ and it is supported in a concentric cube of sidelength equal to $\frac{9}{8} \ell (L)$. 

Denote by $\mathscr{C}_k$ all cubes $L$ of the grid (namely of the subdvision of $[-\sigma, \sigma]^m$ into $2^{km}$ closed cubes of sidelength $2^{-k} \sigma$). We then define the smooth function
\[
\zeta_k (x) = \frac{\sum_{L\in \mathscr{C}_k} \vartheta_L (x) g_L (x)}{\sum_{L\in \mathscr{C}_k} \vartheta_L (x)}
\]
and we call it {\em glued interpolation} at scale $2^{-k}$.

Almgren's theorem is then a simple corollary of the following

\begin{mythm}\label{t:almgren2}
Fix $M_0$ and $N_0$ as above. If $\delta>0$ is sufficiently small, then
there are positive constants $C$, $\beta$ and $\varepsilon_0$ (depending on $m,n, M_0, N_0$ and $\delta$) with the following properties. 
Let $u$ be as in Theorem \ref{t:almgren} and consider for each $k$ the glued interpolation $\zeta_k$ at scale $2^{-k}$. Then,
\begin{itemize}
\item[(a)] $\|D\zeta_k\|_{C^{2, \beta}} \leq C E^{\frac{1}{2}}$;
\item[(b)] $\|\zeta_k - u\|_{C^0 ([-\sigma, \sigma]^m)} \to 0$ as $k\uparrow\infty$.
\end{itemize}
\end{mythm}

Clearly the key estimate in the theorem above is (a), since it is rather obvious that each interpolating function $g_L$
is in fact very close to $u$ on its own domain of definition: (b) is left to the reader as an exercise. 

\subsection{Estimates on the interpolating functions.}  Consider an $L\in \mathscr{C}_k$ with $k>N_0$. There is then a unique cube $K\in \mathscr{C}_{k-1}$ which contains it. $K$ will be called {\em the father of $L$}. Analogously, if $K\supset L$, $L\in \mathscr{C}_k$, $K\in \mathscr{C}_j$ and $j<k$, $K$ will be called {\em an ancestor} of $L$. Finally, if $K, L\in \mathscr{C}_k$ have nonempty intersection, they will be called {\em neighbors}.\footnote{Distinct cubes of $\mathscr{C}_k$ have disjoint interiors: the intersection of two neighbors is therefore a subset of
some $m-1$-dimensional plane and might reduce to a single point.}

The estimate (a) of Theorem \ref{t:almgren2} will then be a consequence of the following ones on the various ``pieces'' which we glue together.

\begin{myprop}\label{p:key_estimates}
Let $u$ be as in Theorem \ref{t:almgren} and the parameters $M_0$ and $N_0$ are fixed as in Theorem \ref{t:almgren2}. Then, provided $\delta$ and $E$ are chosen smaller than appropriate geometric constants, there are $\beta>0$ and $C$ such that the following holds
\begin{itemize}
\item[(i)] If $L\in \mathscr{C}_k$, then
\begin{equation}\label{e:est1}
\sum_{i=1}^3 \|D^i g_L\|_{C^0 (B_{2 M_0 \ell (L)} (x_L))} \leq C E^{\frac{1}{2}}\qquad \|D^4 g_L\|_{C^0} \leq C 2^{(1-\beta)k} E^{\frac{1}{2}}\, .
\end{equation}
\item[(ii)] If $K\in \mathscr{C}_j$ is an ancestor of $L$, then
\begin{equation}\label{e:inter}
\sum_{i=0}^4 2^{(3+\beta-i) j} \|D^i (g_L-g_K)\|_{C^0 (B_{2 M_0 \ell (L)} (x_L))} \leq C  E^{\frac{1}{2}}\, .
\end{equation}
\item[(iii)] If $K, L\in \mathscr{C}_k$ are neighbors, then the estimate \eqref{e:inter} holds for $g_L -g_K$ on its domain of definition
$B_{2M_0 \ell (L)} (x_L) \cap B_{2M(0) \ell (L)} (x_K)$.
\end{itemize}
\end{myprop}

\begin{proof}[Proof of Estimate (a) in Theorem \ref{t:almgren2}] Fix $k \geq N_0$ and define
\[
\theta_L := \frac{\vartheta_L}{\sum_{J\in \mathscr{C}_k} \vartheta_J}\, .
\]
Observe that we have the obvious estimates
\[
\|D^i \theta_L\|_{C^0} \leq C (i) 2^{ik} \qquad \qquad \forall L\in \mathscr{C}_k\, .
\]
Consider now one $L\in \mathscr{C}_k$ and let $\mathscr{N} (L)$ be the set of its neighbors. It is then obvious that
\[
(\zeta_k - g_L) = \sum_{K\in \mathscr{N} (L)} \theta_K (g_K - g_L)\, .
\] 
In particular for every $i \in \{1,2,3\}$ we have\footnote{In the estimates we have used the obvious geometric fact that the cardinality of $\mathscr{N} (L)$ is bounded by a geometric constant $C(m)$.}
\begin{align}
\|D^i \zeta_k\|_{C^0 (L)} &\leq \|D^i g_L\|_{C^0} + \sum_{K\in \mathscr{N} (L)} \|D^i (\theta_K (g_K- g_L))\|_{C^0 (L)}\nonumber\\
&\leq C E^{\frac{1}{2}} + C \sum_{K\in \mathscr{N} (L)} \sum_{0\leq i \leq j} \|D^j \theta_K\|_{C^0} \|D^{i-j} (g_K-g_L)\|_{C^0 (L)}\nonumber\\
&\leq C E^{\frac{1}{2}} +C E^{\frac{1}{2}} \sum_{0\leq i\leq j} 2^{jk} 2^{(i-j-3-\beta)k}\leq C E^{\frac{1}{2}}\, .
\end{align}
In particular 
\[
\|D\zeta_k\|_{C^2} \leq C E^{\frac{1}{2}}\, .
\]
Note that a very similar computations yields
\[
\|D^4 \zeta_k\|_{C^0} \leq C E^{\frac{1}{2}} 2^{(1-\beta) k}\, .
\]
In particular, if $x,y\in L$ we easily conclude
\begin{equation}\label{e:near}
|D^3 \zeta_k (x) - D^3 \zeta_k (y)|\leq |x-y| \|D^4 \zeta_k\|_{C^0}
\leq C E^{\frac{1}{2}} |x-y|^\beta\, .
\end{equation}
Consider next the centers $x_L$ and $x_K$ of two different cubes $K, L\in \mathscr{C}_k$ and assume for the moment that $J$ is the {\em first} common ancestor of both $L$ and $K$.
Observe that $\ell (J)\leq C |x_L-x_K|$. Moreover, by construction $\zeta_k$ equals $g_L$ in a neighborhood of $x_L$ and equals $g_K$ in a neighborhood of $x_K$. We thus can estimate
\begin{align}
& |D^3 \zeta_k (x_L) - D^3 \zeta_k (x_K)| =
|D^3 g_L (x_L) - D^3 g_K (x_K)|\nonumber\\
\leq\; & |D^3 g_L (x_L) - D^3 g_J (x_L)| + |D^3 g_J (x_L) - D^3 g_J (x_K)|\nonumber\\
&\qquad + |D^3 g_J (x_K) - D^3 g_K (x_K)|\nonumber\\
\leq\; & \|D^3 (g_L - g_J)\|_{C^0 (L)} + \sqrt{m}\, \ell (J)\, \|D^4 g_J\|_{C^0 (J)}
+ \|D^3 (g_J- g_K)\|_{C^0 (K)}\nonumber\\
\leq\; & C E^{\frac{1}{2}} \ell (J)^\beta \leq C E^{\frac{1}{2}} |x_L-x_K|^\beta\, .\label{e:far}
\end{align}
Combining \eqref{e:near} with \eqref{e:far} we easily conclude that, for any $j\leq k$ and any cube $M \in \mathscr{C}_j$ we have\footnote{As it is customary in the PDE literature, $[f]_{\alpha, F}$ denotes the H\"older seminorm of the function $f$ on the subset $F$ of its domain, namely
\[
[f]_{\alpha, F} := \sup_{x\neq y, x,y\in F} \frac{|f(x)-f(y)|}{|x-y|^\alpha}\, .
\]}
\[
[D^3 \zeta_k]_{\beta, M} \leq C E^{\frac{1}{2}}\, .
\]
In particular the latter estimate holds for every cube $J\in \mathscr{C}_{N_0}$. Since however $\mathscr{C}_{N_0}$ covers $[-\sigma, \sigma]^m$ and consists of $2^{kN_0}$ cubes, we finally get 
\[
[D^3 \zeta_k]_{\beta, [-\sigma, \sigma]^m}\leq C E^{\frac{1}{2}}\,. 
\]
\end{proof}

\subsection{Changing coordinates.} We will see in the next section that much of the estimates leading to Proposition
\ref{p:key_estimates} will in fact be carried on in the ``tilted'' systems of coordinates. For this reason we will make
heavy use of the following techincal lemma.

\begin{mylem}\label{l:rotation_complex}
There are constants $c_0,C>0$ with the following properties.
Assume that
\begin{itemize}
\item[(i)] $A\in SO(m+n)$, $|A-{\rm Id}|\leq c_0$, $r\leq 1$;
\item[(ii)] $(x_0, y_0)\in\pi_0\times\pi_0^\perp$ are given and $f,g: B^m_{2r} (x_0)\to \mathbb R^n$ are Lipschitz functions such that
\begin{equation*}
\Lip (f), \Lip (g) \leq c_0\quad\text{and}\quad |f(x_0)-y_0|+|g(x_0)-y_0|\leq c_0\, r.
\end{equation*}
\end{itemize}
Then, in the system of coordinates $(x',y')= A (x,y)$, for $(x_1,y_1) = A (x_0,y_0)$, the following holds:
\begin{itemize}
\item[(a)] $\gr (f)$ and $\gr (g)$ are the graphs of two Lipschitz functions $f'$ and $g'$ in the tilted system of coordinates, whose domains of definition contain both $B_{r} (x_1)$;
\item[(b)] $\|f'-g'\|_{L^1 (B_{r} (x_1))}\leq C\,\|f-g\|_{L^1 (B_{2r} (x_0))}$;
\item[(c)] if $f\in C^4 (B_{2r} (x_0))$, then $f'\in C^4 (B_{r} (x_1))$, with the
estimates 
\begin{eqnarray}
\|f'- y_1\|_{C^3}&\leq& \Phi \left(|A-{\rm Id}|, \|f-y_0\|_{C^3}\label{e:est_C3}\right)\, ,\\
\|D^4 f'\|_{C^0} &\leq& \Psi \left(|A-{\rm Id}|, \|f-y_0\|_{C^3}\right) 
\left(1+ \|D^4 f\|_{C^0}\right)\, ,
\label{e:est_C4}
\end{eqnarray}
where $\Phi$ and $\Psi$ are smooth functions.
\end{itemize}
\end{mylem}

\begin{proof} Let $P: \mathbb R^{m\times n}\to \mathbb R^{m}$ and
$Q: \mathbb R^{m\times n}\to \mathbb R^{n}$ be the usual orthogonal projections.
Set $\pi=A(\pi_0)$ and
consider the maps $F, G: B_{2r} (x_0)\to \pi^\perp$ and $I, J: B_{2r} (x_0)\to \pi$
given by
\[
F (x) = Q (x,f(x))\quad\text{and}\quad G(x) = Q (x, g(x)),
\]
\[
I(x)= P (x, f(x))\quad\text{and}\quad J (x) = P (x, g(x)).
\]
Obviously, if $c_0$ is sufficiently small, $I$ and $J$ are injective Lipschitz maps.
Hence, $\gr (f)$ and $\gr (g)$ coincide, in the $(x',y')$ coordinates, with the graphs of the functions
$f'$ and $g'$ defined respectively in $D:= I (B_{2r} (x_0))$ and $\tilde{D}:= J (B_{2r} (x_0))$
by $f' = F \circ I^{-1}$ and $g'= G \circ J^{-1}$.
If $c_0$ is chosen sufficiently small, then we can find a constant $C$
such that 
\begin{equation}\label{e:Lip_bound}
\Lip (I), \; \Lip (J),\; \Lip (I^{-1}),\; \Lip (J^{-1}) \leq 1+C\,c_0,
\end{equation}
and
\begin{equation}\label{e:Linfty_bound}
|I (x_0)-x_1|, |J (x_0)-x_1|\leq C\,c_0\, r.
\end{equation}

Clearly, \eqref{e:Lip_bound} and \eqref{e:Linfty_bound} easily imply (a).
Conclusion (c) is a simple consequence of the inverse function theorem.
Finally we claim that, for small $c_0$,
\begin{equation}\label{e:claim}
|f'(x')-g'(x')|\leq 2 \,|f (I^{-1} (x')) - g (I^{-1} (x'))|
\quad\forall \;x'\in B_r(x_1),
\end{equation}
from which,
using the change of variables formula for biLipschitz homeomorphisms
and \eqref{e:Lip_bound}, (b) follows. Claim \eqref{e:claim} is an elementary exercise
in classical euclidean geometry and it is left to reader.
\end{proof}

\section{$C^{3,\beta}$ estimates}

\subsection{Key estimates on the tilted interpolation.} In this section we finally come to the
core PDE argument which will allow us to derive the estimates of Proposition \ref{p:key_estimates}.
We consider however a slightly more general situation. We fix $L\in \mathscr{C}_k$ but consider any plane $\pi$ such that in the corresponding cylinder $\bC_{16 M_0 \ell (L)} (p_L, \pi)$ we have the estimate\footnote{In particular we have already shown that $\pi_L$ falls in such category, but we will need to consider system of coordinates (and in particular cylinders) where the ``base plane'' $\pi$ might be tilted, compared to $\pi_L$, by the same amount which estimates the tilt between $\pi_L$ and the horizontal plane $\pi_0$.}
\begin{equation}\label{e:plane_OK}
\bar{E} := \ell (L)^{-m} \bE (\gr (u), \bC_{16M_0 \ell (L)} (p_L, \pi)) \leq C E \ell(L)^{2-2\delta}\, .
\end{equation}
We then let $\bar{f}$ be the $\bar{E}^\gamma$-Lipschitz approximation and set $\bar{z} := \bar{f} \ast \varphi_{\ell (L)}$.

\begin{myprop}\label{p:key_prop}
If $\delta$ and $E$ are sufficiently small, there is $\beta >0$ and $C, C (j)$ geometric constants such that
\begin{align}
\|\bar{z} - \bar{f}\|_{L^1 (B_{4 M_0 \ell (L)} (p_L, \pi))} &\leq C E \ell (L)^{m+3+2\beta}\label{e:L1}\\
\|\Delta D^j \bar{z}\|_{C^0 (B_{4 M_0 \ell (L)} (p_L, \pi))} & \leq C (j)E \ell (L)^{1-j+2\beta} \qquad \forall j\in \mathbb N\, .\label{e:Delta-Linf}
\end{align}
\end{myprop}

\begin{proof} In order to simplify the notation we drop $p_L$ and $\pi$ and simply write $B_s$ for $B_s (p_L, \pi)$.
Let $v$ be the function whose graph describes $\gr (u)$ in the $\pi\times \pi^\perp$ coordinates. 
Fix a test function $\kappa\in C^\infty_c (B_{8 M_0 \ell (L)})$ and consider the first variation of the area functional along $\kappa$. By minimality of $v$
\[
0 = [\delta \gr (v)] (\kappa) = \left. \frac{d}{ds}\right|_{s=0} \int \sqrt{1+ |Dv+ s D\kappa|^2 
+ \sum_{|\alpha|\geq 2} M_{\alpha\lambda} (Dv+ sD\kappa)^2} \, .
\]
Moreover, since $v$ and $\bar{f}$ coincide aside from a set of measure no larger than $C E^{1+\gamma} \ell(L)^{m+ (1+\gamma) (2-2\delta)}$, cf. \eqref{e:alessio_richiama_1} and \eqref{e:alessio_richiama_2}, we easily conclude 
\[
|[\delta \gr (\bar f)] (\kappa)|\leq C E^{1+\gamma} \ell(L)^{m+ (1+\gamma)(2-2\delta)}\|D \kappa\|_{C^0} \, .
\]
Finally we use an explicit computation and a simple Taylor expansion\footnote{The Taylor expansion is the expansion of $D_A F$, where 
\[
F (A) = \left(1+|A|^2 + \sum_{|\alpha|\geq 2} (\det M_{\alpha\beta} (A))^2\right)^\frac{1}{2}
\]
is the area integrand. It is then elementary that, if $A_{ij}$ denotes the $ij$-entry of the matrix $A$, then
\[
\partial_{A_{ij}} F (A) = A_{ij} + O (|A|^3)\, .
\]} 
to derive
\begin{align*}
\left|\delta \gr (\bar f) (\kappa) - \int D \bar f : D\bar \kappa\right| &\leq C \int |D\bar f|^3 |D\kappa|
\leq C E^\gamma \ell (L)^{(2-2\delta)\gamma} \|D\kappa\|_{C^0} \int |D\bar f|^2\\
&\leq C E^{1+\gamma} \ell(L)^{m+ (1+\gamma)(2-2\delta)}\|D \kappa\|_{C^0}\, 
\end{align*}
(where $A:B$ denotes the Hilbert-Schmidt scalar product of the matrices $A$ and $B$, namely $A:B = {\rm tr}\, (A^TB)$).

We next impose that $\delta$ is sufficiently small, so that 
\begin{equation}\label{e:define_beta}
(1+\gamma)(2-2\delta) = 2+2\beta
\end{equation} 
for a positive $\beta$. 
In particular we conclude
\begin{equation}\label{e:fondamentale}
\left|\int D \bar f : D\bar \kappa\right| \leq C E^{1+\gamma} \ell(L)^{m+ 2+2\beta}\|D \kappa\|_{C^0}\, .
\end{equation}
The gain $E^\gamma$ is however not important for our considerations and we will therefore neglect it in all our subsequent 
estimates. 

If we denote by $K^c$ the complement of the set over which $\bar f$ and $v$ coincide, 
the same considerations above (the suitable modifications are left to the reader) yield also the estimate
\begin{align}
\int_{B_{7 M_0 \ell (L)}} \left|\int D\bar f (w) : D \psi (w-x)\, dw\right|\, dx &\leq C \|\mathbf{1}_{K^c}* D\psi\|_{L^1}
+ C \| |D\bar{f}|^3 * D \psi\|_{L^1}\nonumber\\
&\leq C \left(|K^c| + \int |Dv|^3\right)\|D\psi\|_{L^1}\nonumber\\
&\leq C E \ell (L)^{m+2+2\beta} \|D\psi\|_{L^1}\label{e:fondamentale2}
\end{align}
for every test $\psi\in C^\infty_c (B_{M_0 \ell (L)})$.

The latter are the fundamental estimates from which \eqref{e:L1} and \eqref{e:Delta-Linf} are (respectively) derived. We start with
\eqref{e:Delta-Linf} considering that
\begin{align*}
\|\Delta \bar{z}\|_{C^0 (B_{4 M_0 \ell (L)})}
&= \sup_{\|\psi\|_{L^1} \leq 1} \left|\int \psi \Delta \bar z\right|
=  \sup_{\|\psi\|_{L^1} \leq 1} \left|\int D \bar f : D (\psi \ast \varphi_{\ell (L)})\right|\\
&\stackrel{\eqref{e:fondamentale}}{\leq} C E \ell(L)^{m+ 2+2\beta} \sup_{\|\psi\|_{L^1}\leq 1} \|D (\psi \ast \varphi_{\ell (L)})\|_{C^0}\\
&\leq C E \ell(L)^{m+ 2+2\beta} \|D\varphi_{\ell (L)}\|_{C^0}
\leq C E \ell(L)^{1+2\beta}\, .
\end{align*}
As for the higher derivative estimates, they are obvious consequences of 
\[
D^j \bar{z} = \bar{f} \ast D^j (\varphi_{\ell (L)})\, .
\]
\eqref{e:L1} is more subtle. Write
\[
\bar z (x)- \bar f(x) = \int \varphi_{\ell (L)} (x-y) (\bar f (y) - \bar f (x))\, dy\, .
\]
In order to simplify our notation assume for the moment $x=0$ and compute
\begin{align*}
\bar{z} (0) -\bar f (0) &= \int \varphi_{\ell (L)} (y) \int_0^{|y|}
\frac{\partial \bar f}{\partial r} \left( \tau \frac{y}{|y|}\right)\, d\tau\, dy\\
&= \int \varphi_{\ell (L)} (y) \int_0^{|y|} \nabla \bar f \left(\tau\frac{y}{|y|}\right)
\cdot \frac{y}{|y|}\, d\tau\, dy\notag\\
&= \int \varphi_{\ell (L)} (y) \int_0^1 \nabla \bar f (\sigma y) \cdot y\, d\sigma\, dy\\
&= \int \int_0^1 \varphi_{\ell (L)} \left(\frac{w}{\sigma}\right)\, \nabla \bar f (w) \cdot
\frac{w}{\sigma^{m+1}}\, d\sigma\, dw\notag\\
&= \int \nabla \bar f (w) \cdot \underbrace{w \left(\int_0^1 \varphi_{\ell (L)}
\left(\frac{w}{\sigma}\right)\,\sigma^{-m-1}\, d\sigma\right)}_{=: \Phi (w)}\, dw .
\end{align*}
More generally, 
$\bar{z} (x) -\bar f (x)= \int \nabla \bar f (w) \cdot \Phi (w-x)\, dw$
and
\begin{equation*}
\|\bar z-\bar f\|_{L^1 (F)}=\int_F \left\vert \int \nabla \bar f (w) \cdot \Phi (w-z)\, dw\right\vert\,dz.
\end{equation*}
Since $\varphi$ is radial, the function $\Phi$ is a gradient.
Indeed, it can be easily checked that, for any $\psi$, the vector field $\psi (|w|)\, w$ is curl-free.
Moreover, the support of $\Phi$ is compactly contained in $B_{\ell (L)}$.
Thus we can apply \eqref{e:fondamentale2} to derive
\begin{equation}\label{e:norm1}
\|\bar z - \bar f\|_{L^1(B_{4 M_0 \ell (L)})}\leq C\, E^{1+\gamma} \ell (L)^{m+2+2\beta} \|\Phi\|_{L^1}\, .
\end{equation}
Since
\begin{align*}
\|\Phi\|_{L^1} \leq \int \int_0^1 |w|\,\varphi
\left(\frac{w}{\ell (L) \sigma}\right)\,\ell (L)^{-m}\sigma^{-m-1}\, d\sigma\,dw
=\ell (L) \int_0^1 \int  |y| \,\varphi (y)\,dy\,d\sigma\, ,
\end{align*}
we easily conclude $\|\Phi\|_{L^1} \leq C \ell(L)$, which we can insert in \eqref{e:norm1} to conclude the proof.
\end{proof}

\subsection{Proof of estimate (i) in Proposition \ref{p:key_estimates}} Consider now $L\in \mathscr{C}_k$ as in the statement of the Proposition and let $L = L_k \subset L_{k-1} \subset \ldots \subset L_{N_0}$ be its ``ancestry''.
Fix the plane $\pi = \pi_L$ and fix a natural number $j \in [N_0, k-1]$. Recall that, by the De Giorgi's excess decay, 
$|\pi_L - \pi_{L_{j}}|\leq C E \ell (L_{j})^{2-2\delta}$: arguing as for estimating $\bE (\gr (u), \bC_{16 M_0 \ell (L)} (p_{L_j}, \pi_{L_j}))$,
we easily conclude that 
\[
\ell (L_j)^{-m} \bE (\gr (u), \bC_{16 M_0 \ell (L_j)} (p_{L_j}, \pi_L))\leq C E \ell (L_j)^{2-2\delta_2}\, .
\]
Hence we are in the position of applying the estimate of Proposition \ref{p:key_prop} in the cylinder 
\[
\bC^j = \bC_{8 M_0 \ell (L_j)} (p_{L_j}, \pi)\, .
\] 
If $\bar{f}_j$ are the corresponding Lipscitz approximations and $\bar{z}_j = \bar{f}_j \ast \varphi_{\ell (L_j)}$, we then conclude from Proposition \ref{p:key_prop}
\begin{align*}
\|\Delta \bar{z}_j\|_{C^0} &\leq C E \ell (L_{j})^{1+2\beta}\\
\|\bar{z}_j - \bar{f}_j\|_{L^1} &\leq C E \ell (L_{j})^{m+3+2\beta}\, .
\end{align*}
Consider now two consecutive maps $\bar z_j$ and $\bar z_{j+1}$. The domain of definition of the second map is $B = B_{4 M_0 \ell (L)_{j+1}} (p_{L_{j+1}}, \pi)$ and is contained in the domain of definition 
of the first. Moreover, recalling \eqref{e:alessio_richiama_1}, in such common domain $B$ both $\bar{f}_j$ and $\bar{f}_{j+1}$ coincide with the same function 
(and hence they are equal) except for a set of measure at most $C E^{1+\gamma} \ell (L_j)^{m+2+2\beta}$. In particular, they coincide on a nonempty set and, since they are both Lipschitz, $\|\bar{f}_j - \bar{f}_{j+1}\|_{C^0 (B)} \leq C \ell (L_j)$. Combining the latter two bounds we are able to estimate the $L^1$ norm of $\bar{f}_j - \bar{f}_{j+1}$. We thus conclude the two estimates
\begin{align}
\|\Delta (\bar{z}_j - \bar{z}_{j+1})\|_{C^0 (B)} &\leq C E \ell (L_{j})^{1+2\beta}\label{e:interpol10}\\
\|\bar{z}_j - \bar{z}_{j+1}\|_{L^1 (B)} &\leq C E \ell (L_{j})^{m+3+2\beta}\label{e:interpol11}\, .
\end{align}
We leave to the reader the proof that, classical estimates for the Laplacian and classical interpolation inequalities (cf. for instance \cite{Nir}), imply then
\begin{equation}\label{e:interpol13}
\|D (\bar{z}_j - \bar{z}_{j+1})\|_{C^0 (B')} \leq C E \ell (L_{j})^{2+2\beta}\, ,
\end{equation}
for $B' = B_{\frac{7}{2} M_0 \ell (L_j)}$.
We can next estimate $\|D^i (\bar{z_j} - \bar{z}_{j+1}\|_{C^0}$ for $i>1$ interpolating between \eqref{e:interpol13} and 
\begin{equation}\label{e:interpol12}
\|\Delta D^{i-1} (\bar{z}_j - \bar{z}_{j+1})\|_{C^0 (B)} \leq C E \ell (L_{j})^{2 - i +2\beta}\, ,
\end{equation}
where the latter follows from the estimates of the higher derivatives of the Laplacian given in Proposition \ref{p:key_prop}. We can thus conclude that for $i\in \{0, 1\ldots , 4\}$ and in $B'' =B_{3 M_0 \ell (L_j)}$
\[
\|D^i (\bar{z}_j - \bar{z}_{j+1})\|_{C^0 (B'')} \leq C E \ell (L_j)^{3+2\beta-i} \leq C E 2^{-(3+2\beta -i)j}\, .
\]
By interpolation we achieve then
\[
[D^3(\bar{z}_j - \bar{z}_{j+1})]_{\beta, B''} \leq C E \ell(L_j)^\beta\, .
\]
Summing the corresponding geometric series and recalling that $z_L = \bar z_k$ we easily achieve 
\begin{align*}
\|D^i (z_L - \bar{z}_{N_0})\|_{C^0 (B_{3 M_0 \ell (L)} (p_L, \pi_L))} &\leq C E \qquad \mbox{for $i\in \{0,1,2,3\}$,}\\
[D^3 (z_L - \bar{z}_{N_0})]_{\beta, B_{3 M_0 \ell (L)} (p_L, \pi_L)} &\leq C E\, ,
\end{align*}
and
\[
\|D^4 (z_L - \bar{z}_{N_0})\|_{C^0 (B_{3 m_0 \ell (L)} (p_L, \pi_L))} \leq C E \ell (L)^{2\beta -1} \leq C E 2^{(1-2\beta) k}\, .
\]
Now, since $\bar{z}_{{N_0}}$ is the convolution at a scale comparable to $1$ of a Lipschitz function $\bar{f}_{N_0}$ with
$\|D \bar{f}_{N_0}\|_{L^2} \leq C E^{\frac{1}{2}}$, we clearly have
\[
\|D \bar{z}_{N_0}\|_{C^s} \leq C(s) E^{\frac{1}{2}} \qquad \mbox{for every $s\in \mathbb N$.}
\]
We therefore easily conclude\footnote{Indeed the proof gives the better exponent $1-2\beta$ in the bound for the $C^4$ norm.} 
\begin{align*}
\|D z_L\|_{C^{2, \beta} (B_{3M_0 \ell (L)} (p_L, \pi))} &\leq C E^{\frac{1}{2}}\\
\|D^4 z_L\|_{C^0 (B_{3M_0 \ell (L)} (p_L, \pi))} &\leq C E^{\frac{1}{2}} 2^{(1-\beta) k}\, . 
\end{align*}
Using now Lemma \ref{l:rotation_complex} we achieve estimate (i) in Proposition \ref{p:key_estimates}. 

\subsection{Proof of the estimates (ii) and (iii) of Proposition \ref{p:key_estimates}} First of all we observe that in
order to show (ii) in full generality, it suffices to show it when $K$ is the father of $L$ and then sum the corresponding estimates over the relevant ancestry of $L$ in the general case. As such, the estimates (ii) and (iii) are then very similar
and they can in fact be proved using the same idea. The key is again a suitable $L^1$ estimate.

\begin{mylem}\label{l:interaction}
Assume $u$ and $\delta$ satisfy the assumptions of Proposition \ref{p:key_prop} and consider a pair of cubes $(K,L)$ which consists of either father and son or of two neighboring cubes of the same cubical partition $\mathscr{C}_k$. If $F$ is the intersection of the domain of definitions of $g_L$ and $g_K$, then
\begin{equation}\label{e:interaction}
\|g_L - g_K\|_{L^1 (F)}\leq C E \ell (L)^{m+3+\beta}\, .
\end{equation}
\end{mylem}

The lemma will be proved below and we now show how to derive the estimates (ii) and (iii) from it. First observe that the case of fourth derivatives in \eqref{e:inter} is an obvious consequence of the estimate \eqref{e:est1}. As for the the other derivatives, observe that we know from Part (i) of Proposition \ref{p:key_prop} that
$\|g_L-g_K\|_{C^{3, \beta} (F)} \leq C E^{\frac{1}{2}}$. Combining the latter bound with \eqref{e:interaction} we then conclude the proof applying the following lemma.

\begin{mylem}\label{l:interpolation_bis}
For every $m$, $0<r<s$ and $\kappa>0$ there is a positive constant $C$ (depending on $m$, $\kappa$ and $\frac{s}{r}$)
with the following property. 
Let $f$ be a $C^{3, \kappa}$ function in the disk $B_s\subset \mathbb R^m$, taking values in $\mathbb R^n$. Then
\begin{equation}\label{e:interpolation_bis}
\|D^j f\|_{C^0 (B_r)} \leq C r^{-m-j} \|f\|_{L^1 (B_s)} + C r^{3+\kappa-j} [D^3 f]_{\kappa, B_s}\qquad \forall j\in \{0,1,2,3\}\, ,
\end{equation}
where $C$ is a constant depending only on $m$, $n$ and $\kappa$. 
\end{mylem}
\begin{proof} A simple covering argument reduces the lemma to the case $s=2r$. Moreover,
define $f_r (x):= f (rx)$ to see that we can assume $r=1$ and, arguing componentwise, we can assume $n=1$. 
So our goal is to show
\begin{equation}\label{e:ridotta}
\sum_{j=0}^3 |D^j f (y)| \leq C \|f\|_{L^1} + C [D^3 f]_\kappa\, \qquad \forall y \in B_1, \forall
f\in C^{3,\kappa} (B_2, \mathbb R)\, ,
\end{equation}
By translating it suffices then to prove the estimate
\begin{equation}\label{e:ridotta_2}
\sum_{j=0}^3 |D^j f (0)| \leq C \|f\|_{L^1 (B_1)} + C [D^3 f]_{\kappa, B_1} \qquad \forall f\in C^{3,\kappa} (B_1)\, .
\end{equation}
Consider now the space of polynomials $R$ in $m$ variables of degree at most 3, which we write as
$R (x) = \sum_{0 \leq |j|\leq 3} A_j x^j$, where we use the convention that:
\begin{itemize}
\item $j= (j_1, \ldots, j_m)$ denotes a multiindex;
\item $|j| = j_1+ \ldots + j_m$;
\item $x^j = x_1^{j_1} x_2^{j_2} \ldots x_m^{j_m}$.
\end{itemize} 
This is a finite dimensional vector space, on which we can define the norms
$|R| := \sum_{0\leq |j|\leq 3} |A_j|$ and $\|R\| := \int_{B_1} |R(x)|\, dx$.
These two norms must then be equivalent, so there is a constant $C$ (depending only on $m$), such 
that $|R|\leq C\|R\|$ for any such polynomial. In particular,
if $P$ is the Taylor polynomial of third order for $f$ at the point $0$, we conclude
\begin{align*}
\sum_{j=0}^3 |D^j f (0)| &\leq C |P| \leq C \|P\| = C \int_{B_1} |P (x)|\, dx
\leq C\|f\|_{L^1 (B_1)} + C\|f-P\|_{L^1 (B_1)}\\
&\leq C \|f\|_{L^1} + C [D^3 f]_\kappa\, .\qquad\qquad \qedhere
\end{align*}
\end{proof}

In order to complete our task we are only left with proving Lemma \ref{l:interaction}.

\begin{proof}[Proof of Lemma \ref{l:interaction}]
Since the two cases are analogous, we consider the one in which $K$ is the father of $L$. We let $z_L$ and $z_K$ be the corresponding tilted interpolating functions, which come from the convolutions of the functions $f_L$ and $f_K$. Now, the graph of $z_K$ is contained in the cylinder $\bC_{4 M_0 \ell (K)} (p_K, \pi_K)$. We can however apply Lemma \ref{l:rotation_complex} and find functions $\hat{f}$ and $\hat{z}$ defined on $B_{3 M_0 \ell (K)} (p_K, \pi_L) \to \pi_L^\perp$ with the properties that 
\begin{align*}
\gr (\hat{z}) &= \gr (z_K) \cap \bC_{3 M_0 \ell (K)} (p_K, \pi_L)\\
\gr (\hat{f}) &= \gr (f_K) \cap \bC_{3 M_0 \ell (K)} (p_K, \pi_L)\, . 
\end{align*}
Now, by Lemma \ref{l:rotation_complex} we have
\[
\|\hat{z} - \hat{f}\|_{L^1 (B_{3 M_0 \ell (K)} (p_K, \pi_L))}
\leq \|z_K - f_K\|_{L^1 (B_{4M_0 \ell (K)} (p_K, \pi_K))} \leq C E \ell (K)^{m+3+\beta}\, ,
\]
where we have used Proposition \ref{p:key_prop} in the last inequality.

Consider now that $B = B_{4M_0 \ell (L)} (p_L, \pi_L)$, which is the domain of $z_L$, is contained in $B_{3 M_0 \ell (K)} (p_K, \pi_L)$. 
Moreover, both $\gr (\hat{f})$ and $\gr (f_L)$ coincide with $\gr (u)$ except for a set of $m$-dimensional volume bounded by $C \ell (K)^{m+2+\beta}$, cf.  \eqref{e:alessio_richiama_1}. In particular, 
\[
|\{\hat{f} \neq f_L\}\cap B| \leq C E \ell (K)^{m+2+\beta}\, ,
\] 
and thus the two functions agree on a nonempty set. Since they are both Lipschitz with bounded Lipschitz constant, 
\[
\|\hat{f}-f_L\|_{L^1 (B)}\leq C E \ell (K)^{m+3+\beta}\, .
\]
We can thus use again Proposition \ref{p:key_prop} to infer
\[
\|\hat{z} - z_L\|_{L^1 (B)} \leq C E \ell (K)^{m+3+\beta}\, .
\]
Now, consider that 
\begin{align*}
\gr (g_L) & = \gr (z_L) \cap \bC_{2 M_0 \ell (L)} (x_L, \pi_0)\\
 \gr (g_K) \cap \bC_{2 M_0 \ell (L)} (x_L, \pi_0) &= \gr (\hat{z}) \cap \bC_{2 M_0 \ell (L)} (x_L, \pi_0)\, .
 \end{align*} 
 Thus we can use Lemma \ref{l:rotation_complex} one last time to derive
\[
\|g_K-g_L\|_{L^1 (B_{2M_0 \ell (L)} (x_L, \pi_0))} \leq C \|\hat{z} - z_L\|_{L^1 (B)} \leq C E \ell (K)^{m+3+\beta}
\]
and conclude the proof of the lemma.
\end{proof}

%
%
%
%
%
%
%

\bibspread


\end{document}